\RequirePackage{tikz}
\documentclass[default]{sn-jnl}
\author*{\fnm{Takahiro} \sur{Yamada}}\email{g.yamadatakahiro@gmail.com}
\affil{\orgdiv{The Department of Philosophy and Religious Studies}, \orgname{Utrecht University}, \orgaddress{\street{Janskerkhof 13}, \city{Utrecht}, \postcode{3512 BL}, \state{Utrecht}, \country{the Netherlands}}}
\usepackage{tikzit_for_Springer}
\tikzstyle{small}=[fill=black, draw=black, shape=circle, inner sep=0pt, minimum size=1pt]
\tikzstyle{thick}=[-, line width=0.8pt]
\tikzstyle{semithick}=[-, line width=0.6pt]
\tikzstyle{thin}=[-, line width=0.4pt]
\tikzstyle{verythin}=[-, line width=0.2pt]
\tikzstyle{rounded corners very thin}=[-, line width=0.2pt, rounded corners=1pt]
\tikzstyle{rounded corners arrow}=[->, line width=0.2pt, rounded corners=1pt]

\usepackage{comment}
\usepackage[normalem]{ulem}
\usepackage{amsmath}
\usepackage{bussproofs}
\usepackage{multicol}
\usepackage{multirow}

\usepackage{xcolor}
\usepackage{soul}

\allowdisplaybreaks[1]

\newtheoremstyle{mystyle1}
  {}
  {}
  {\normalfont}
  {}
  {\bfseries}
  {}
  {10pt}
  {}
\theoremstyle{mystyle1}
\newtheorem{definition}{Definition} 
\newtheorem{principle}[definition]{Principle} 

\newtheoremstyle{mystyle2}
  {}
  {}
  {\itshape}
  {}
  {\bfseries}
  {}
  {10pt}
  {}
\theoremstyle{mystyle2}
\newtheorem{proposition}[definition]{Proposition} 
\newtheorem{corollary}[definition]{Corollary} 
\newtheorem{lemma}[definition]{Lemma} 
\newtheorem{claim}[definition]{Claim} 

\DeclareMathOperator{\PA}{\textup{\textbf{PA}}}
\DeclareMathOperator{\SF}{\textup{\textbf{SF}}}
\DeclareMathOperator{\SFQ}{\textup{\textbf{SFQ}}}
\DeclareMathOperator{\SFQp}{\textup{\textbf{SFQ}}_\textup{\textbf{P}}}
\DeclareMathOperator{\NSF}{\textup{\textbf{NSF}}}
\DeclareMathOperator{\NSFp}{\textup{\textbf{NSF}}_\textup{\textbf{P}}}

\DeclareMathOperator{\CPC}{\textup{\textbf{CPC}}}
\DeclareMathOperator{\CQC}{\textup{\textbf{CQC}}}

\DeclareMathOperator{\IPC}{\textup{\textbf{IPC}}}
\DeclareMathOperator{\IQC}{\textup{\textbf{IQC}}}
\DeclareMathOperator{\IQCE}{\textup{\textbf{IQCE}}}
\DeclareMathOperator{\HT}{\textup{\textbf{HT}}}

\DeclareMathOperator{\HTCD}{\textup{\textbf{HTCD}}}

\DeclareMathOperator{\fin}{fin}

\DeclareMathOperator{\Var}{Var}
\DeclareMathOperator{\FV}{FV}

\DeclareMathOperator{\Func}{Func}
\DeclareMathOperator{\Term}{Term}
\DeclareMathOperator{\ClTerm}{ClTerm}
\DeclareMathOperator{\Pred}{Pred}
\DeclareMathOperator{\Atom}{Atom}
\DeclareMathOperator{\ClAtom}{ClAtom}
\DeclareMathOperator{\Fm}{Form}
\DeclareMathOperator{\ClForm}{ClForm}
\DeclareMathOperator{\Cont}{Cont}
\DeclareMathOperator{\ClCont}{ClCont}

\DeclareMathOperator{\GN}{GN}
\DeclareMathOperator{\ClGN}{ClGN}

\DeclareMathOperator{\ST}{ST}

\newcommand{\wneg}{{\sim}} 

\begin{document}

\setcounter{tocdepth}{7}
\setcounter{secnumdepth}{5} 

\title{Wright's First-Order Logic of Strict Finitism}

\abstract{A classical reconstruction of Wright's first-order logic of strict finitism is presented. Strict finitism is a constructive standpoint of mathematics that is more restrictive than intuitionism. Wright sketched the semantics of said logic in \cite{Wright1982}, in his strict finitistic metatheory. Yamada \cite{Yamada2023} proposed, as its classical reconstruction, a propositional logic of strict finitism under an auxiliary condition that makes the logic correspond with intuitionistic propositional logic. In this paper, we extend the propositional logic to a first-order logic that does not assume the condition. We will provide a sound and complete pair of a Kripke-style semantics and a natural deduction system, and show that if the condition is imposed, then the logic exhibits natural extensions of Yamada \cite{Yamada2023}'s results.
}



\def\fCenter{\ \vdash\ } 

\maketitle

\section{Introduction: Aims}\label{section: Background and aims}
We reconstruct in the classical metatheory, and explore, Wright's first-order logic of strict finitism, the ideas of which were sketched in \cite{Wright1982}. `Strict finitism' in Wright's and our sense is a constructive standpoint of mathematics that is more restrictive than intuitionism. An intuitionist accepts a statement if it is verifiable in principle; and a number if it is (mentally) constructible in principle. The strict finitist's philosophical tenet is that `in principle' should be replaced by `in practice'. Then the activity of mathematics, according to them, is based on our actual human cognitive capacities. The legitimate, `strict finitistic', numbers are those actually constructible. Their totality cannot be (classically) infinite, since we could not actually construct infinitely many objects.

We hold, after some authors\footnote{
Cf. e.g. \cite[p.113]{Wright1982}, \cite[p.147]{Tennant1997} and \cite[p.278]{Murzi2010}.
},
that the notion of possibility in principle in this context can be identified with possibility in practice with some extension. Particularly concerning verifiability, we assume the following.
\begin{principle}[Connection between intuitionism and strict finitism]\label{principle: Conceptual identity}
  A statement is verifiable in principle iff it is verifiable in practice with some finite extension of practical resources.
\end{principle}
\noindent We regard this principle as connecting intuitionism and strict finitism in one fundamental sense. Also, we assume the same for constructibility of a number.

The logic of strict finitism is meant to be the abstract system of reasoning acceptable according to strict finitism. Throughout this article, we may use the phrase `strict finitistic logic' (and similar others) to mean said logic, as well as its classical reconstruction -- just as the modern usage of the phrase `intuitionistic logic'.

Since Bernays's strict finitistic doubt cast on Brouwer in as early as 1934 \cite[p.265]{Bernays1934}, a considerable amount of research has been done\footnote{
An overview of related research until 1995 is Cardone's \cite[pp.2-3]{Cardone1995}. We note Van Bendegem's works under the name `strict finitism', e.g. \cite{VanBendegem1994} and \cite{VanBendegem2012}. But we cannot pursue the relation between his strict finitism and Wright's.
}.
Amongst the literature, Wright \cite{Wright1982} interests us the most, as it provided a systematic explanation of strict finitistic reasoning, i.e. a semantics, for the first time. His paper was a philosophical reply to Dummett \cite{Dummett1975}'s criticism which loosely targeted Yessenin-Volpin's ideas in e.g. \cite{Yessenin-Volpin1970}\footnote{
See Isles' \cite{Isles1980} for a theoretical reconstruction of Yessenin-Volpin's ideas.
}.
Wright's semantics was model-theoretic, formed in his strict finitistic metatheory. Yamada \cite{Yamada2023} in 2023 reconstructed it in the classical metatheory and presented propositional logic $\SF$ of Wrightian strict finitism, by incorporating one assumption called the `prevalence condition'\footnote{
We note that in \cite{Yamada2023}, `$\SF$' is the name of his sequent calculus.
}. 
$\SF$ exhibits a strong similarity to intuitionistic propositional logic $\IPC$. Namely, finite models of $\SF$ can be viewed as nodes in a model of the Kripke semantic of $\IPC$, and Yamada proved a result that appears to formalise the bridging principle (\ref{principle: Conceptual identity}) between intuitionism and strict finitism.

We extend $\SF$ and present $\SFQ$, strict finitistic first-order logic, in the classical metatheory without the prevalence condition. Its semantics can be viewed as a variant of the intuitionistic Kripke semantics. Just as in \cite{Yamada2023}, a model is a tree-like structure that represents all possible histories of a modelled agent's actual verification steps. The forcing conditions of quantification are kept unchanged from \cite{Wright1982} as much as possible. A sound and complete pair of the semantics and a natural deduction system is provided. We also conduct a study of $\SFQp$, $\SFQ$ with the prevalence condition imposed. In most aspects it exhibits natural extensions of the properties of $\SF$. Finite models of $\SFQp$ can be viewed as nodes in a model of the Kripke semantics of intuitionistic first-order logic $\IQC$; and we provide a quantificational version of the aforementioned formalisation result of the bridging principle (\ref{principle: Conceptual identity}).

We hope our research may corroborate Yamada \cite{Yamada2023}'s direction of classical formalisation of strict finitistic reasoning. Certainly, strict finitism has been received as an elusive standpoint. If our attempt is successful, however, strict finitistic reasoning can be understood in the classical framework, and in one sense is indeed connected to intuitionistic reasoning.

In section \ref{section: Methods: Our rendition of Wright's semantic}, we discuss how Wright defines the semantic system, what strict finitistic notions are used there, and how we classically render them in order to reconstruct the semantics. In essence, we render the totality of the strict finitistic numbers as set $\mathbb{N}$ of the classical natural numbers; and we will freely use induction on the meta-level in discussing the system of strict finitistic reasoning, as opposed to a common belief that strict finitism rejects induction. This issue is by nature philosophical. The mathematical part starts in section \ref{section: Strict finitistic first-order logic}.


\section{Methods: Our rendition of Wright's semantics}\label{section: Methods: Our rendition of Wright's semantic}

This section is to provide conceptual preparations for the mathematical part. We explain how we render Wright's notions, and note some limitations on the present article's scope. Although we do not hold that our rendition is philosophically implausible, we never claim that it is entirely innocent. Our arguments shall be loaded with conceptual assumptions which one might as well object with reasons. But, since the cogency of our attempt shall hinge on how reasonable the rendition is, it would behoove us to explicitly state in what sense and on what grounds our logic is a classical reconstruction of Wright's logic. We must leave fully philosophical discussions for another occasion\footnote{
We are planning a full philosophy article on this subject.
},
and only hope that this section contributes to explicating the conceptual assumptions of our classically-reconstructed strict finitistic logic from section \ref{section: Strict finitistic first-order logic} onwards.

In section \ref{section: The semantic system in the strict finitistic framework}, we describe Wright's system in his strict finitistic framework. Next in section \ref{section: The conceptual constraints} explained are the characteristic conceptual constraints on the objects involved. This part may serve as a better description of strict finitism. Then, in section \ref{section: Our rendition and its supports}, we state how we render the core of the constraining properties. Some philosophical principles are introduced there in order to support, or at least to explicitly state the grounds of, our rendition.

\subsection{The semantic system in the strict finitistic framework}\label{section: The semantic system in the strict finitistic framework}
$\blacksquare$ \textit{Wright's language} \quad Wright conceived of his semantic system with respect to a first-order arithmetical language. Let us assume that it is a standard one that consists of the symbols $\underline{0}$ for zero, $S$ for the successor function, $+$ for addition, $\cdot$ for multiplication and $=$ for the identity. He calls the numerals, with formation rules assumed, the `natural-number denoting expressions' (or `NDEs'\footnote{
Wright \cite[p.167]{Wright1982} used `nde' for the singular and `nde's' for the plural form.
}).
We take the NDEs to be the terms of the language, and the atomic sentences the equations between them. It may be natural for Wright to focus on arithmetic, as strict finitism is primarily concerned with numbers. In this article, while we develop a general system regarding an arbitrary language, we cannot investigate arithmetical theories.

$\blacksquare$ \textit{A frame and the forcing conditions} \quad Wright's semantics uses the notion of a `strict finitistic tree' whose nodes are called `elementary arithmetical accumulations'. We suggest imagining a tree-like frame $\langle {K, \leq \rangle}$ of the intuitionistic semantics\footnote{
We spell out the definition of an intuitionistic frame in section \ref{section: Semantics}.
}.
An accumulation is a representation of a `logically possible state of human information' and if two accumulations $k$ and $k'$ stand in the order-relation $k \leq k'$, then $k'$ is practically attainable from $k$ via information-gathering \cite[p.167]{Wright1982}. The targeted, modelled agent is thought to be a single human agent or a team of humans \cite[p.169]{Wright1982}. With each accumulation $A$, on our simplification, is associated a pair $\langle{ M_A, E_A \rangle}$. Wright calls $M_A$ an `elementary notational scheme'. It is the set of the NDEs thought to have been (practically) constructed by (i.e. up to and including) $A$. $E_A$ stands for the basic arithmetical statements that have been verified by $A$. It is a set of equations made only of the elements of $M_A$ \cite[p.167]{Wright1982}.

A strict finitistic tree $T_A^*$ for $A$ consists of the sequences $\Sigma$ of accumulations that satisfy the following.
\begin{enumerate}
  \item For any component $B$ in $\Sigma$, $M_A \subseteq M_B$ and $E_A \subseteq E_B$.
  \item $\Sigma$ has $A$ as its initial element.
  \item $\Sigma$ is well-ordered by the subset relation ($\subseteq$).
  \item For any $B$ in $\Sigma$, the immediate successor $C$ of $B$ satisfies that either $M_C \backslash M_B$ is exactly one element, or so is $E_C \backslash E_B$.
\end{enumerate}
From (ii) and (iii), we can say that $T_A^*$ is, in the modern terms, a tree with $A$ as its root. Wright defines the actual verification (i.e. forcing) conditions of a statement at each accumulation $B$ in $T_A^*$ (cf. \cite[pp.167-9]{Wright1982})\footnote{
Wright calls them the `verification-conditions' \cite[p.169]{Wright1982}. But we prefer the phrase `actual verification' in \cite[section 3]{Wright1982}.
}. The symbolism in the following is ours, including $\models$ for the forcing relation. For an atomic $P$, $B \models P$ if $P$ belongs to $E_B$, and for any $Q$ and $R$,
\begin{enumerate}
  \item $B \models Q \land R$ if $B \models Q$ and $B \models R$,
  \item $B \models Q \lor R$ if $B \models Q$ or $B \models R$,
  \item $B \models \neg Q$ if there is no $\Sigma$ that has a $C$ with $C \models Q$,
  \item $B \models Q \to R$ if for any $\Sigma$ that has $B$, and for any $C$ in $\Sigma$, if $C \models Q$, then there is a $\Sigma'$ that has $B$, $C$ and some $D$ with $D \models R$,
  \item $B \models \exists x Fx$ if there is an $n \in M_B$ such that $B \models Fn$ and
  \item $B \models \forall x Fx$ if for any NDE $n$, any $\Sigma$ that has $B$, and any $C$ in $\Sigma$, if $n \in M_C$, then there is a $\Sigma'$ that has $B$, $C$ and some $D$ with $D \models Fn$.
\end{enumerate}
See section \ref{section: Semantics} for our reconstruction.

\subsection{The conceptual constraints}\label{section: The conceptual constraints}
Every object is subject to strict finitistic restrictions, or the entire system is so designed. The wording depends on the object: any NDE standing for a number has to be `actually intelligible' (or `actually constructible'); statement `actually verifiable'; proof `surveyable'; set `manageable' etc. Also, `practically' can be used interchangeably with `actually'. However, we can assume they all mean the same kind of human feasibility and intuitive intelligibility, for their respective objects. We will describe some restrictions regarding actual constructibility, which are the properties we will render, but assume the same for all the other notions, too.

$\blacksquare$ \textit{Actual weak decidability} \quad Actual constructibility is said to be `actually weakly decidable'. Namely, for any NDE, there is a humanly feasible programme of investigation whose implementation is bound to produce at least ground for asserting either that we can actually construct it or not (cf. \cite[pp.133-5, pp.160-2, esp. p.133]{Wright1982}). Simply put, or so we suggest, if one imagines writing down an NDE, they would obtain an opinion on whether the process of writing down can be finished or not.

$\blacksquare$ \textit{The soritical conditions} \quad Wright writes about the following conditions of the so-called `sorites paradox' across the paper. We formulate in term of NDEs. 
\begin{principle}[The soritical conditions on actual constructibility]
  \quad
  \begin{enumerate}
    \item [($I$)] $\underline{0}$ is actually constructible (the initial condition).
    \item [($T$)] For any actually constructible NDEs $m$ and $n$, $S(n)$ is actually constructible; and similar for $m + n$ and $m \cdot n$ (the tolerance condition).
    \item [($D$)] For any actually constructible NDE $n$, if $n$ is $S(n')$ for some NDE $n'$, then $n'$ is actually constructible; and similar for the cases with $n$ being $n' + n''$ or $n' \cdot n''$ (the downward closure).
    \item [($X$)] There is a number for which there is no actually constructible NDE (the existence of an upper bound).\footnote{
    He did have the subject of the sorites paradox in mind: cf. \cite[p.156]{Wright1982}. ($T$) plays an important role in the discussion of a `Dedekindian account' of infinity \cite[p.130-1]{Wright1982}; and a justification of ($T$) can be found on \cite[p.165]{Wright1982}. Concerning ($X$), he says even a `complete' system does not have $10^{10^{10}}$ actually constructible numerals (cf. \cite[pp.156-7]{Wright1982}).
    }
  \end{enumerate}
\end{principle}
\noindent Particularly for ($D$), he argues that a strict finitist may prefer a `complete' (intuitively, gapless) system with a maximal expressive power: cf. \cite[pp.156-7]{Wright1982}. ($X$) suggests that the scope of discourse involves NDEs that are not actually constructible. 
As noted above, similar conditions are assumed for any other strict finitistic predicates.

`Wang's paradox', the title of \cite{Dummett1975}, refers to a variety of arguments that deduce a contradiction from these conditions. While they are not contradictory for a strict finitist -- rather fundamental --, apparently ($I$), ($T$) regarding $S$ and ($X$) are inconsistent, if the induction principle on the conceptual level is used. Thus Yessenin-Volpin (\cite[p.4]{Yessenin-Volpin1970}), Dummett (\cite[p.251]{Dummett1975}) and Wright (cf. \cite[pp.154-5, ]{Wright1982}) assume that strict finitism does not allow for induction. Accordingly, in \cite{Wright1982}, induction on the complexity of a statement is avoided, and Wright admits that he assumes with no justification that the forcing relation persists in a strict finitistic tree \cite[p.170]{Wright1982}.

Even without assuming induction, however, one might argue that a contradiction is yielded by a chain of universal instantiation of ($T$) and Modus Ponens. Dummett presented an argument regarding `apodicticity' (cf. section \ref{section: Induction in strict finitism}) in this regard. While we cannot fully address arguments to this effect here, we endorse the `neo-feasibilist' approach by Dean \cite{Dean2018} in section section \ref{section: N for the strict finitistic numbers} as one solution.

Wright calls `indefinite' the magnitude of the totality of the objects subject to the soritical conditions. For instance, a sequence $\Sigma$ (as a set) in $T_A^*$ is of indefinite length (cf. \cite[p.165, p.168]{Wright1982}); and he adds `$\Sigma$ is a humanly practicable sequence of information-gathering' as a fifth property of $\Sigma$ (cf. section \ref{section: The semantic system in the strict finitistic framework}).

In Wright's conception, therefore, there is no leaf accumulation in a strict finitistic tree. We hold that he was considering a tree containing all possible expansions of an information state. But for a general semantics, this view may be too restrictive. We endorse the view that a strict finitistic tree is one theoretical representation (i.e. model) of all possible information states of an agent; and a sequence in it is a possible history of the actual verification of the agent, and it can be `definite' in the sense that it has its last component. An agent may stop constructing in some history.

$\blacksquare$ \textit{Non-guaranteed `prevalence'} \quad Wright writes about one worry. Suppose that there is a $\Sigma$ in $T_A^*$ that has a $B$ with $B \models P$ for some atomic $P$. Here, there is no reason to assume that every sequence $\Sigma'$ has a $B'$ such that $B' \models P$. This is conceptually because verifying $P$ may involve `too much work' and might not be verified after some others. Accordingly, $P$ and $Q$ being each verified at some accumulation does not imply that both $P$ and $Q$ being verified at the same accumulation. He takes this as a threat to the `validity of conjunction-introduction' \cite[p.168]{Wright1982}\footnote{
Cf. the forcing condition of conjunction in section \ref{section: The semantic system in the strict finitistic framework}.
}.

The issue itself is interesting, but Wright's writing here is not most clear. He defines as follows in the context of a $T_A^*$.
\begin{definition}[Wright's assertibility and validity]
  A sentence $P$ is `assertible' if there is a $\Sigma$ in $T_A^*$ that has an accumulation $B$ with $B \models P$; and `valid' if for any $\Sigma$ in $T_A^*$, $\Sigma$ has a $B$ with $B \models P$. \cite[p.170]{Wright1982}
\end{definition}
\noindent `Assertibility' may translate to satisfiability. But validity thus defined is not common. We endorse reserving `validity' for being verified at every accumulation; and using the word `prevalence' for Wright's `validity', since a sentence is then verified almost everywhere. Then what is worried about is the fact that the assertibility of $P$ and that of $Q$ do not imply even the assertibility (not `validity' in his sense) of $P \land Q$. We note that surely, if $Q$ is further prevalent, then $P \land Q$ becomes assertible, assuming that $P$ persists\footnote{
We note that what he considers as `logic' in this context is of the Hilbert style; and the axiom corresponding to conjunction introduction is $P \to (Q \to (P \land Q))$ \cite[p.171]{Wright1982}. This does not precisely capture the situation he describes where $P$ and $Q$ are separately verifiable, given the forcing condition of implication in section \ref{section: The semantic system in the strict finitistic framework}. Our definitions of validity and assertibility appear in section \ref{section: Basic properties and the semantic notions}.
}.

We agree that an agent needs to be ideal in some sense to be able to verify every assertible sentence in every possible history. However, rejecting this notion altogether would be too restrictive. One or two statements might be easy enough to be prevalent. Also, it would be of technical and conceptual interest to see how the logic behaves under the `prevalence condition' that every atomic formula is prevalent. \cite{Yamada2023} incorporated it and showed that the reconstructed logic has a strong similarity to $\IPC$.

\subsection{Our rendition and some supports}\label{section: Our rendition and its supports}
Bold though it may be, we treat the totality of the numbers actually representable by an NDE as if it were set $\mathbb{N}$ of the classical natural numbers. Accordingly, the indefinite magnitude is going to be classical countable cardinality $\aleph_0$; and the length of a sequence in a strict finitistic tree, at most $\omega$. Induction on the concepts that originate in strict finitism will also be used in the reconstruction.

On the one hand, this rendition has a support in the form of formal results. Under this rendition, desirable results formally follow, if we interpret the definitions of Wright's semantic system relatively straightforwardly. The actual weak decidability in section \ref{section: The conceptual constraints} is obtained as that $\neg A \lor \neg \neg A$ is valid for any formula $A$ (lemma \ref{lemma: several famous validities}, iv). A formalisation of the bridging principle (\ref{principle: Conceptual identity}) between intuitionism and strict finitism follows, under the prevalence condition and some seemingly meaningful restriction on the models (proposition \ref{corollary: Th GC is IQC}). We submit especially the latter as an important coincidence.

On the other hand, this rendition may be supported conceptually, too, in a way of the `neo-feasibilist' approach to the sorites paradox by Dean \cite{Dean2018}.

\subsubsection{$\mathbb{N}$ for the strict finitistic numbers}\label{section: N for the strict finitistic numbers}
$\blacksquare$ \textit{The neo-feasibilist approach} \quad The core idea of neo-feasibilism is that the soritical conditions in section \ref{section: The conceptual constraints} can be consistent on Peano arithmetic $\PA$, with induction regarding the language of arithmetic. The key here is to regard the upper bound in ($X$) as unspecified. Namely, extend the language with $P$ for actual constructibility, and formalise as follows: ($I$) $P(\underline{0})$, ($T$) $\forall x (P(x) \to P(S(x)))$, ($D$) $\forall x \forall y (P(x) \land y < x \to P(y))$ and ($X$) $\exists x (P(x) \to \bot)$, where $\bot$ stands for the falsum\footnote{
The details of $<$ do not matter. One may introduce symbol $<$ with axioms, or define $x < y$ to be $\exists z (x + Sz = y)$.
}.
Then, formal induction in $\PA$ does not apply to $P$, since $P$ is additional to the original language; and a countable nonstandard model of $\PA$ can be extended to satisfy them\footnote{
See \cite[p.315]{Dean2018} for a proof. He discusses using an arithmetical theory that includes $\PA^-$, instead of $\PA$ itself, where $\PA^-$ is $\PA$ without the induction scheme.
}.
Also, since the upper bound is unspecified, the arguments from universal instantiation of ($T$) and Modus Ponens are avoided (cf. section \ref{section: The conceptual constraints}).

The structure of a countable nonstandard model is known to be order-isomorphic to $\mathbb{N} \cup (\mathbb{Z} \times \mathbb{Q})$, intuitively lined in the following way. It starts with $\mathbb{N}$; and above it, nonstandard numbers are lined just as $\mathbb{Q}$ (the rationals); and for each of them, there are adjacent nonstandard numbers on the right and on the left, ordered just as $\mathbb{Z}$ (the integers)\footnote{
See e.g. \cite[pp.74-5]{Kaye1991} for a more precise definition of the structure and a proof of this fact.
}.
In this interpretation, $P$ receives as its extension a \textit{cut}, i.e. an initial segment closed under the successor function on the structure; and the upper bound in ($X$) is interpreted as a nonstandard number outside of it. Dean endorses taking a cut that is not $\mathbb{N}$ as the extension of $P$, assuming that a predicate in the sorites paradox is vague\footnote{
Dean holds that this way, one can deal with `higher-order vagueness' \cite[pp.330-1, fn.63]{Dean2018}, and attain a better philosophical position than `epistemic ism' thanks to Tennenbaum's theorem \cite[pp.332-4.]{Dean2018}.
}.

We agree regarding the upper bound. In the context of strict finitism, the upper bound would be a number for which no NDE is actually constructible. The scope of discourse consists of the numbers which are actually constructible, and those which are not.\footnote{
One might not think that this idea suits strict finitism, a constructivist view. A number might need to be actually constructible, in order to enter the scope of discourse. This is to reject condition ($X$), and one might try to avoid Wang's paradox this way. Indeed, Mawby considered a version of strict finitism in this line of thought, `fanatical strict finitism' \citep[pp.183-95]{Mawby2005}. However, we cannot pursue further here.
}

However, we do not have strong reasons to avoid $\mathbb{N}$ for the extension of $P$. This is because a predicate of strict finitism is not vague according to Wright \cite[p.132]{Wright1982}: it is actually weakly decidable (cf. section \ref{section: The conceptual constraints}). Rather, we endorse taking $\mathbb{N}$ as the extension, since the numbers genuinely acceptable to a strict finitist do not seem to exceed classical finitude, as we describe in the following.

$\blacksquare$ \textit{Canonical NDEs as the basis} \quad We call, after Wright, `canonical' the NDEs of the form $S \cdots S(\underline{0})$ \cite[cf. p.174]{Wright1982}. We take the view that the strict finitistic numbers are based on them. A strict finitist would accept a number iff it is actually constructible (cf. section \ref{section: Background and aims}). In terms of NDEs, then, for a number to be accepted, at least one of its NDEs must actually be constructible. We suggest that this can be sharpened via Mawby's principle. In his dissertation \cite{Mawby2005} dedicated to the study of strict finitism, he proposes the following.
\begin{principle}[Mawby's criterion]
  A number is accepted iff its magnitude can intuitively be grasped.
\end{principle}
\noindent This magnitude principle comes from the idea that we recognise a mere string of symbols as standing for a number only when we can precisely grasp the magnitude of the denoted number \cite[Cf. pp.57-8, p.74.]{Mawby2005}. We hold that the magnitude of a natural number can be displayed by `how many $S$'s occur before $\underline{0}$', and endorse the following.
\begin{principle}[Displayed magnitude]
  The magnitude of a number can intuitively be grasped iff its canonical NDE is actually constructible.
\end{principle}
\noindent This way, we treat numbers primarily in terms of the canonical NDEs. An additive or multiplicative NDE $m$ stands for a number as long as it is `actually convertible', in the sense that there is a canonical NDE $c$ such that $m = c$ is actually verifiable\footnote{
However, again, we cannot pursue in this article what formal arithmetic would result in from this view.
}.

$\blacksquare$ \textit{Feasibility is finite} \quad We hold that anything denotable via a canonical NDE must, under a classical rendition, be finite. This is because any actual construction of an NDE must be feasible, and feasibility shall not exceed finitude. Any canonical NDE $S \cdots S (\underline{0})$, if actually constructible, is reachable in a feasible way: one can write down all of $\underline{0}, S(\underline{0}), ..., S \cdots S(\underline{0})$ in this order by adding $S$ to the left. In this sense, $S \cdots S (\underline{0})$ is the end product of a sequence of definite length. So we suggest its referent should not be rendered as a nonstandard number, but to an $n \in \mathbb{N}$, since for any nonstandard number, there is an infinite descending sequence of nonstandard numbers below it. 

\subsubsection{Induction in strict finitism}\label{section: Induction in strict finitism}
We endorse the induction principle in the strict finitistic framework, with the restriction on the quantification in the conclusion, to the numbers actually representable by a canonical NDE. This may fit our neo-feasibilist rendition of the totality of those numbers as $\mathbb{N}$. Surely, formal induction does not apply to $P$ (cf. section \ref{section: N for the strict finitistic numbers}). But, for all $n \in \mathbb{N}$, $P(\overbrace{S \cdots S}^{n} \underline{0})$ is true in the model, by formalised ($I$) and ($T$).

This restriction, on the one hand, nullifies the reason to suspend induction, concerning the soritical conditions in section \ref{section: The conceptual constraints}. If the conclusion only covers the actually constructible numbers, one can let the upper bound be unspecified (cf. section \ref{section: N for the strict finitistic numbers}), and take a number not actually constructible to be it.

Also, this may fit Wright's forcing condition of universal quantification in section \ref{section: The semantic system in the strict finitistic framework}. His idea is to quantify only over the actually constructible objects in the future. Then it may be reasonable to assume that the scope of quantification is thus restricted\footnote{
To be precise, our reconstruction involves `global quantification' that quantifies over all the things in the scope of discourse. We hold that Wright's forcing condition of negation induces this mode of quantification. See section \ref{section: The basic semantic definitions}, especially its remarks.
}.

On the other hand, we suggest a strict finitist might have a justification for induction with this restriction. Below, we will argue that induction is justifiable regarding the `apodictic' numbers, and then propose that every strict finitistic number is apodictic.

$\blacksquare$ \textit{Induction on apodictic numbers} \quad Dummett \cite[p.9]{Dummett2000} and Magidor \cite[p.475]{Magidor2012} describe how an intuitionist would justify induction, for an arbitrary predicate $P'$. It is intuitively that the conclusion $\forall m P'(m)$ is guaranteed by the premises $P'(0)$ and $\forall n (P'(n) \to P'(n+1))$, since for any $m$, $P'(m)$ is established by $m$ repeated applications of universal instantiation and Modus Ponens. As a constructivist, a strict finitist would have no objection, if the derivation is surveyable. But Magidor holds that they would reject induction, thinking that some $m$ may be too large.

Here, let us consider `apodicticity', Dummett's version of surveyability (cf. section \ref{section: The conceptual constraints}). A number $n$ is `apodictic' if any derivation within $n$ steps is small enough, in the sense that one can recognise it to be following the rules of inference \cite[cf. p.253]{Dummett1975}\footnote{
We would agree that apodicticity only looks at the vertical length, and the number of symbols involved should also be taken into account. However, we expect that a notion that covers all relevant aspects of a surveyable derivation shall obey the soritical conditions, and an argument to the same effect applies.
}.
This is a strict finitistic notion on numbers. Apodicticity is subject to the soritical conditions in section \ref{section: The conceptual constraints}, and the totality of the apodictic numbers is included in that of the strict finitistic numbers.

Then, with details fine-tuned, the induction principle may be justifiable to a strict finitist in the way Dummett and Magidor describe, if the conclusion is restricted to the apodictic numbers. Namely, for any predicate $P'$ and apodictic $m \, (\geq 1)$, one could establish $P'(m)$ from $P'(0)$ and $\forall n (P'(n) \to P'(n+1))$ within $m+1$ steps. Whenever $m$ is apodictic, so is $m+1$. Therefore the derivation is ex hypothesi small enough. 

$\blacksquare$ \textit{Actual constructibility implies apodicticity} \quad We claim the following.
\begin{claim}
  Every actually constructible number is apodictic.
\end{claim}
\noindent To argue for this, we appeal to Yessenin-Volpin's principle, the `central ontological hypothesis' (`c.o.h.'). In the present context, it guarantees the actual constructibility of a sequence on some conditions.
\begin{principle}[The central ontological hypothesis]\label{principle: The central ontological hypothesis}
  Let $n$ be a strict finitistic number. Suppose that a sequence with $n$ steps is defined to satisfy the following:
  \begin{enumerate}
    \item the first step is actually constructible, and
    \item for each step, if the sequence so far is actually constructible, then so is the sequence up to the next step.
  \end{enumerate}
  Then the whole sequence is actually constructible. \cite[Cf. pp.28-9]{Yessenin-Volpin1970}
\end{principle}
\noindent In other words, the end product of a sequence is guaranteed to be actually constructible on `inductive' grounds, if the length of the sequence is known to be strict finitistic\footnote{
In Yessenin-Volpin's original context, mathematics is viewed as based on eleven kinds of actions, including construction in his way \cite[p.16]{Yessenin-Volpin1970}. A sequence is essentially a prescribed process of actions, although he is not explicit. The c.o.h is introduced to guarantee the feasibility of an entire procedure. In his formulation, our items (i) and (ii) are not separated.
}.

Now, suppose $n$ is strict finitistic. Let us use $Q$ for apodicticity. First, by ($I$) and ($T$) regarding apodicticity, we have $Q(2)$. We define a sequence with $n$ steps as follows, The $0$-th component is
\begin{prooftree}
  \AxiomC{$Q(2)$}
  \AxiomC{$\forall x (Q(x) \to Q(x+1))$}
  \UnaryInfC{$Q(2) \to Q(3)$}
  \BinaryInfC{$Q(3)$.}
\end{prooftree}
The next is 
\begin{prooftree}
  \AxiomC{$Q(2)$}
  \AxiomC{$\forall x (Q(x) \to Q(x+1))$}
  \UnaryInfC{$Q(2) \to Q(3)$}
  \BinaryInfC{$Q(3)$}
  \AxiomC{$\forall x (Q(x) \to Q(x+1)$}
  \UnaryInfC{$Q(3) \to Q(4)$}
  \BinaryInfC{$Q(4)$.}
\end{prooftree}
The final is
{\small 
\begin{prooftree}
  \AxiomC{$Q(2)$}
  \AxiomC{$\forall x (Q(x) \to Q(x+1))$}
  \UnaryInfC{$Q(2) \to Q(3)$}
  \BinaryInfC{$Q(3)$}
  \AxiomC{$\forall x (Q(x) \to Q(x+1))$}
  \UnaryInfC{$Q(3) \to Q(4)$}
  \BinaryInfC{$Q(4)$}
  \noLine
  \UnaryInfC{$\vdots$}
  \noLine
  \UnaryInfC{$Q(n+2)$.}
\end{prooftree}
}
\noindent These are meant to be the semi-formal representations of the informal inferences involving informal expressions only\footnote{
We holds that a strict finitist would understand a semi-formal representation if every numeral and ellipsis is expressed by a canonical NDE; and this sequence can pass the criterion. But we do not pursue this point. Also, for a classical reconstruction of these derivations, see section \ref{section: A proof system NSF}.
}.
The $m$-th component's length is $m+2$, and its conclusion informally says that $m+3$ is apodictic.

We apply the c.o.h. (principle \ref{principle: The central ontological hypothesis}). The sequence consists of $n$ steps, and the $0$-th component is actually constructible indeed. If the $m$-th component is actually constructible, then it establishes that $m+3$ is apodictic. Therefore, the next, $m+1$-st component is actually constructible, since its length is $m+3$. Therefore by the c.o.h., the entire sequence is actually constructible, and hence the end product is actually constructible. This implies that the conclusion of the end product is actually verifiable by a derivation: i.e., that $n+2$ is apodictic. So by ($D$) concerning apodicticity, $n$ is apodictic.

\section{Strict finitistic first-order logic}\label{section: Strict finitistic first-order logic}

Upon the rendition discussed in section \ref{section: Our rendition and its supports}, we classically reconstruct Wright's semantic systems (cf. section \ref{section: The semantic system in the strict finitistic framework}). Current section \ref{section: Strict finitistic first-order logic} is the first half where we investigate its formal behaviours in general. In the latter half (section \ref{section: Theory of prevalence}), we investigate the logic's behaviours under the `prevalence property' (cf. sections \ref{section: The conceptual constraints} and \ref{section: Prevalent models}). We end this article with some reflexions in section \ref{section: Ending remarks: Further topic} on the formal results of section \ref{section: Theory of prevalence}.

In section \ref{section: Language}, we set the syntactic definitions. The important class `$\GN$' of the `global negative' formulas is introduced, which is used in most parts of this article. We provide our formal semantics in section \ref{section: Semantics}, with some remarks. A natural deduction system $\NSF$ is introduced in section \ref{section: A proof system NSF}. Section \ref{section: Completeness} is dedicated to its completeness with respect to the strict finitistic semantics. Its proof is in the Henkin-style, not radically different from the case of $\IQC$\footnote{
See e.g. \cite[pp.169-72]{vanDalen2013} and \cite[pp.87-9]{TroelstravanDalen1988}.
}. 
We utilise countably many trees isomorphic to the canonical model of $\IQC$.

\subsection{Language}
\label{section: Language}
Throughout this article, we use a fixed language $\mathcal{L}$ of \textbf{SFQ}. It is a standard language of first-order predicate logic with a distinguished predicate $E$ (the `existence predicate'). We use it to distinguish the constructed objects in the domain of discourse: see section \ref{section: Semantics}.

The language consists of the constant formulas $\top$ and $\bot$, the connectives $\land, \lor, \to, \neg, \exists$ and $\forall$, variables $v_0, v_1, ... \in \Var[\mathcal{L}]$, individual constants $a_0, a_1, ...$, function symbols $f_0, f_1, ... \in \Func[\mathcal{L}]$, predicate symbols $P_0, P_1, ... \in \Pred[\mathcal{L}]$ including unary $E$ and the parentheses. The terms (with metavariables $t, s, ...$), the atomic formulas ($P, Q, ...$) and the formulas ($A, B, ...$) together with the subterms and the subformulas etc. are defined as usual. The classes of all terms, all atomic formulas and all formula are $\Term[\mathcal{L}], \Atom[\mathcal{L}], \Fm[\mathcal{L}]$, respectively. We write $\ClTerm[\mathcal{L}]$ etc. for the classes of the closed terms etc. We abbreviate $A \to \bot$ by $\wneg A$. The parentheses may be omitted as usual, with the usual order of precedence. We define the class $\GN[\mathcal{L}]$ of the \textit{global negative} formulas ($N, N_0, N_1, ...$) by
\begin{itemize}
  \item $\GN[\mathcal{L}] ::= \bot \mid (\neg A) \mid (N \land N) \mid (N \lor N) \mid (N \to N) \mid (\forall x N) \mid (\exists x N)$,
\end{itemize}
where $A \in \Fm[\mathcal{L}]$. The class of the closed GN formulas is $\ClGN[\mathcal{L}]$. Any term $t$ which is not already bound \textit{occurs globally in $A$} if $t$ occurs in a GN subformula of $A$ 
. We note that a term can have a global and a non-global occurrence at the same time: e.g. in $P(t) \lor \neg P(t)$. The only GN formulas in the $\neg$-less fragment are those built only from $\bot$ and the connectives; no terms occur globally in them.

We will use other languages in addition to fixed $\mathcal{L}$. But when no confusion arises, we may drop `$[\mathcal{L}]$' etc. By $\mathcal{L}(D)$, we mean $\mathcal{L}$ whose constants are extended by nonempty set $D$ disjoint with $\mathcal{L}$'s constants. For any $d \in D$, $\overline{d}$ is the \textit{name} of $d$, and we let $\overline{D} = \{ \overline{d} \mid d \in D \}$ for any $D$.

\subsection{Semantics}\label{section: Semantics}
Let us state the basic definitions first, and put some remarks on them (section \ref{section: The basic semantic definitions}). Then we move to some basic properties, with the definitions of several semantic notions, such as validity, assertibility and prevalence (section \ref{section: Basic properties and the semantic notions}).

\subsubsection{The basic semantic definitions}\label{section: The basic semantic definitions}
$\blacksquare$ \textit{The definition of the semantics} \quad Since the definition of a strict finitistic model is based on that of an intuitionistic model, we state the latter in detail. An \textit{intuitionistic frame} (\textit{with respect to $\mathcal{L}$)} is a tuple $\langle K, \leq, D, J \rangle$ with the following properties. $\langle K, \leq \rangle$ is a nonempty partially ordered set. $D$ is a function that assigns to each $k \in K$ a nonempty set $D(k)$ as its `domain', such that $k \leq k'$ implies $D(k) \subseteq D(k')$. $J$ is a family $\langle \langle J_k^1, J_k^2 \rangle \rangle_{k \in K}$ of interpretation function pairs with $J_k^1: \ClTerm[\mathcal{L}(\overline{D(k)})] \to D(k)$ and $J_k^2(f): D(k)^p \to D(k)$ for each $f \in \Func$ with arity $p$ for extended languages. They satisfy the following: (i) $k \leq k'$ implies that $J_{k'}^1 (c) = J_k^1(c)$ for all $c \in \ClTerm[\mathcal{L}(\overline{D(k)})]$, and $J_k^2 (f) \subseteq J_{k'}^2 (f)$, (ii) $J_k^1 (\overline{d}) = d$ for all $k$ and $d \in D(k)$ and (iii) $J_k^1(f(c)) = (J_k^2(f)) ( J_k^1(c) )$ and similar for any other arities. An \textit{intuitionistic model $W = \langle K, \leq, D, J, v \rangle$} (\textit{with respect to $\mathcal{L}$}) is an intuitionistic frame equipped with a valuation $v$ that assigns to each $P \in \Pred$ its extension $P^{v(k)} \subseteq D^n$ at a given node $k \in K$, where $n$ is the arity of $P$, so that for any $P$ and $k, k' \in K$, if $k \leq k'$, then $P^{v(k)} \subseteq P^{v(k')}$. (Thus, $v$ is a function of the type $K \to (\Pred \to \bigcup_{i < \omega} \mathcal{P}(D^i)$.)

Throughout this article, we only consider rooted tree-like frames with at most countable branchings and height at most $\omega$. We may denote a root by $r$. Also, for brevity, we identify $\mathcal{L}(\overline{D(k)})$ with $\mathcal{L}(D(k))$ and similarly for any other cases; and function symbols and predicates are always treated as if they were unary, as long as no confusion arises.

A \textit{strict finitistic frame} is an intuitionistic frame whose $D$ is constant (`constant domain frame'). When no confusion arises, we treat $D$ simply as a nonempty set. Similarly, a pair $\langle J^1, J^2 \rangle$ of interpretation functions with $J^1: \ClTerm[\mathcal{L}(D)] \to D$ and $J^2(f): D^p \to D$ suffices as $J$, if (i) $J^1(\overline{d}) = d$ for each $d \in D$ and (ii) $J^1(f(c)) = (J^2(f)) ( J^1(c) )$. We may write $J$ for both $J^1$ and $J^2$. A \textit{strict finitistic model} is an intuitionistic model with a strict finitistic frame that satisfies the following.
\begin{itemize}
  \item The strictness condition: $\langle J(c_1), ..., J(c_n) \rangle \in P^{v(k)}$ implies $\langle J(c_i') \rangle \in  E^{v(k)}$ for any $P \in \Pred$, $k \in K$, $i \leq n$ and subterm $c_i'$ of $c_i$.
  \item The finite verification condition: $\{ P \mid P^{v(k)} \neq \emptyset \}$ is finite for all $k$.
\end{itemize}
\noindent We call the class of all the strict finitistic models (with respect to $\mathcal{L}$) $\mathcal{W}$. We note that $E^{v(k)}$ may not be closed under $J(f)$ for all $f$.

Fix one $W = \langle K, \leq, D, J, v \rangle \in \mathcal{W}$. We interpret Wright \cite{Wright1982}'s forcing conditions as follows (cf. section \ref{section: The semantic system in the strict finitistic framework}).
\begin{definition}[Actual verification relation $\models_W$]\label{definition: Actual verification relation models_W}
  Let $k \in K$. $k \models_W \top$ and $k \not \models_W \bot$. For any $c \in \ClTerm[\mathcal{L}(D)]$ and $P \in \Pred$, $k \models P(c)$ iff $\langle J(c) \rangle \in P^{v(k)}$. For any $A, B \in \ClForm[\mathcal{L}(D)]$,
  \begin{enumerate}
      \item $k \models_W A \land B$ iff $k \models_W A$ and $k \models_W B$,
      \item $k \models_W A \lor B$ iff $k \models_W A$ or $k \models_W B$,
      \item $k \models_W A \to B$ iff for any $k' \geq k$, if $k' \models_W A$, then there is a ${k'' \geq k'}$ such that $k'' \models_W B$,
      \item $k \models_W \neg A$ iff for any $l \in K$, $l \not \models_W A$,
      \item 
      \begin{enumerate}
	\item if $x$ occurs in $A$ only globally, then $k \models_W \forall x A$ iff for all $d \in D$, $k \models_W \top \to A[\overline{d}/x]$,
	\item otherwise, $k \models_W \forall x A$ iff for any $d \in D$, $k \models_W E(\overline{d}) \to A[\overline{d}/x]$,
      \end{enumerate}
      \item 
      \begin{enumerate}
	\item if $x$ occurs in $A$ only globally, then $k \models_W \exists x A$ iff there is a $d \in D$ such that $k \models_W A[\overline{d}/x]$ and
	\item otherwise, $k \models_W \exists x A$ iff there is a $d \in D$ such that $k \models_W E(\overline{d}) \land A[\overline{d}/x]$.
      \end{enumerate}
  \end{enumerate}
  For an open $A$, we let $k \models_W A$ if $k \models_W A^*$, where $A^*$ is the universal closure (bounded in the order of the free variables' indices). We may drop the subscript $_W$ when no confusion arises.
\end{definition}
\noindent We call quantification according to item (a) `global quantification'; and that according to (b) `local quantification'. We note that these modes depend only on how $x$ occurs in $A$, not in a larger context: e.g., while $\neg \exists x P(x)$ with $P(x) \in \Atom$ itself is a GN formula, the quantification is local.

$\blacksquare$ \textit{Six conceptual remarks} \quad (i) The finite verification condition means that only finitely many predicates are forced at one node. We set this in an attempt to capture that in Wright's system, the set $E_A$ of the atomic sentences actually verifiable at an accumulation $A$ is always manageable (cf. section \ref{section: The semantic system in the strict finitistic framework})\footnote{
We admit that \cite{Yamada2023}'s finite verification condition in the propositional context may be more faithful. It is that $\{ p \in \Var \mid k \in v(p) \}$ is finite, where $\Var$ is the set of the propositional variables. But we did not succeed in proving the completeness (cf. section \ref{section: Completeness}) with the condition e.g. that $P^{v(k)}$ is finite for all $P$ and $k$.
}.

(ii) The forcing condition stands for the verifiability of a statement, but it does not always reflect the modelled agent's knowledge. We suggest the following reading of $k \models A$.
\begin{enumerate}
  \item [(a)] If $A$ is atomic, then it means that the agent has actually verified that $A$ by step $k$, or that they know that $A$ at $k$.
  \item [(b)] If $A$ is complex, then it means that we are committed to that $A$ is correct in light of the agent's actual verifiability.
\end{enumerate}
This division of the perspective results mainly from the forcing condition of negation. $k \models \neg A$ shall naturally mean that $A$ is not verified anywhere: $A$ is unverifiable in practice. But we do not hold that the agent is necessarily aware at step $k$ that $A$ is practically unverifiable. The agent might be trying hard at $k$ to establish that $A$. It is the creator of the model who already knows that $A$ is not verifiable. A model is a representation of all possible information states of a modelled agent (cf. section \ref{section: The semantic system in the strict finitistic framework}), but we suggest that some formulas reflect the entire configuration of the states, which could surpass the agent's recognition.

(iii) In fact, strict finitistic negation $\neg A$ is `global' in the sense that it has a modelwide content. $k \models \neg A$ for some $k$ implies that $k \models \neg A$ for all $k$. So $\neg A$ is about the agent's verificatory capacities as a whole, not about specific facts they have established firsthand by $k$.\footnote{
In the atomic case, on the other hand, there seems to be no difficulty in attributing the knowledge to the agent. In fact, one might extend the attribution to the cases of all formulas without $\neg$, $\to$ and $\forall$, but we do not pursue this issue here.
}

(iv) Implication is in contrast forward-looking: it is intuitionistic implication with an arbitrary finite time-gap. We suggest considering the nodes $k$ of a strict finitistic frame all `within reach' or `in the foreseeable future' in the sense that the extension $P^{v(k)}$ of $P$ at $k$ is already defined. Thus $A \to B$ shall mean `$B$ comes soon after $A$'. We note that $k \models \wneg A$ has the same condition as intuitionistic `local' negation. In contrast to $\neg A$, the total unverifiability of $A$, $\wneg A$ would correspond to that `$A$ is not practically verifiable in the future of this point'.

(v) When considering quantification, Wright used the set $M_A$ etc. (cf. section \ref{section: The semantic system in the strict finitistic framework}). This is to keep track of the constructed objects, and quantify only over them. Thus, in reconstructing, we drew upon $\IQCE$, a variant of $\IQC$ first introduced by Scott \cite{Scott1977} that takes into account construction by adopting the existence predicate $E$ (cf. \cite[pp.50-6]{TroelstravanDalen1988} and \cite{BaazIemhoff2006}). $E$ indicates what is constructed in the domain of discourse: cf. the neo-feasibilist approach in section \ref{section: N for the strict finitistic numbers}. Formally, we let all nodes in a frame have a constant domain $D$, and require $E$ for the terms over which are quantified. The strictness condition also is used in $\IQCE$ \cite[p.52]{TroelstravanDalen1988}.

(vi) However, given negation's global nature, we do not require $E$ for the terms occurring in a negated formula. For example, $\neg P(a)$ ($P(a)$ atomic) can hold at any node even without $E(a)$; and $\exists x \neg P(x)$ should only mean that something in the domain of discourse is never verified to be $P$. So we defined the $\GN$ formulas and a term's `global occurrence' in section \ref{section: Language}, and set two modes of quantification, one that requires $E$, and the other that does not. Thus our quantification is the mixture of that of $\IQC$ and that of $\IQCE$: the `local' conditions are taken from \cite{BaazIemhoff2006}. We note that universal quantification also incorporates a time-gap, as $k \models T \to A$ and $k \models A$ are not equivalent. 


\subsubsection{The semantic notions and basic properties}\label{section: Basic properties and the semantic notions}
Now, let us introduce semantic notions, such as validity and assertibility (cf. section {\ref{section: The conceptual constraints}}), and see some basic properties of this semantics. We start with those of quantification.

The order of the universal closure involved above is not important. For, consider an open $A$ with $x, y \in \FV(A)$ and $x \neq y$. Then $x$ occurs in $A$ only globally iff so it does in $\forall y A$. Therefore it makes sense to speak of how a term occurs in a formula ignoring the outer quantifications. Then by induction, $k \models \forall x_1 ... \forall x_i \forall x_{i+1} ... \forall x_p A$ iff $k \models \forall x_1 ... \forall x_{i+1} \forall x_i ... \forall x_p A$, regardless of how the variables occur. (Also, the same holds for $\exists$.) 

Let $\vec{x}$ be any sequence $\langle x_1, ..., x_p \rangle$ ($1 \leq p$) of variables. We abbreviate $\forall x_1 ... \forall x_p A$ by $\forall \vec{x} A$; similarly for $\exists$. Then we can generalise forcing conditions (v) and (vi) above to $\forall \vec{x} A$ and $\exists \vec{x} A$, respectively, as follows. Let $\vec{y} = \langle x_1, ..., x_q \rangle$ (with $0 \leq q \leq p$) be the exact sequence of the variables locally quantified in $\forall \vec{x} A$ (as well as $\exists \vec{x} A$); and $\vec{z} = \langle x_{q+1}, ..., x_p \rangle$. For convenience, we will be identifying any sequence $\langle d_1, ..., d_n \rangle \in D^n$ ($n \in \mathbb{N}$) with the sequence $\langle \overline{d_1}, ..., \overline{d_n} \rangle$ of their names. Also, for any sequence $\vec{t} = \langle t_1, ..., t_n \rangle$ of terms, we write $A[\vec{t}/\vec{x}]$ for $A[t_1 / x_1, ..., t_n/x_n]$, the result of simultaneous substitution. Then, by induction,
\begin{itemize}
  \item [(v')]
  \begin{itemize}
    \item [(a)] if $q=0$, then $k \models \forall \vec{x} A$ iff for all $\vec{d'} \in D^p$, $k \models \top \to A[\vec{d'} / \vec{z}]$, and
    \item [(b)] if $1 \leq q$, then $k \models \forall \vec{x} A$ iff for all $\vec{d} = \langle d_1, ..., d_q \rangle \in D^q$ and $\vec{d'} \in D^{p-q}$, $k \models \bigwedge_{1 \leq i \leq q} E(\overline{d_i}) \to A[\vec{d} / \vec{y}, \vec{d'} / \vec{z}]$,
  \end{itemize}
  \item [(vi')]
  \begin{itemize}
    \item [(a)] if $q=0$, then $k \models \exists \vec{x} A$ iff for some $\vec{d'} \in D^p$, $k \models A[\vec{d'} / \vec{z}]$, and
    \item [(b)] if $1 \leq q$, then $k \models \exists \vec{x} A$ iff for some $\vec{d} = \langle d_1, ..., d_q \rangle \in D^q$ and $\vec{d'} \in D^{p-q}$, $k \models \bigwedge_{1 \leq i \leq q} E(\overline{d_i}) \land A[\vec{d} / \vec{y}, \vec{d'} / \vec{z}]$. 
  \end{itemize}
\end{itemize}

Now, plainly $\models$ persists. Also, $k \models QxA$ iff $k \models Qy (A [y/x])$, where $Q \in \{ \forall, \exists \}$ and $y$ does not occur in $A$\footnote{
We will generalise this fact in lemma \ref{lemma: Bound variables are always unique}.
}. 
Since by strictness, $k \models P(c)$ implies $k \models E(c')$ for all subterms $c'$ of $c$, our atomic formulas always speak of objects that `exist', `have already been constructed' or are `available to the modelled agent'. However this does not apply to complex formulas: consider e.g. $P(c) \to Q(c)$, and $\neg E(c)$. On the other hand, since vacuous quantification is considered local, $k \models \exists x P(c)$, say, implies $k \models E(\overline{d})$ for some $d$. Observe that $\neg$ and $\wneg$ differ significantly. $k \models \neg A$ implies $k \models \wneg A$, but not vice versa; and $k \models \exists x \neg P$ does not always imply $k \models \exists x \wneg P$, where $P \in \Atom$.

We define the notion of validity as usual, beside that of `assertibility' (cf. section \ref{section: The conceptual constraints}).
\begin{definition}[Validity, assertibility]\label{definition: Validity, assertibility, semantic consequence}
  $A \in \Fm[\mathcal{L}(D)]$ is \textit{valid in $W$} (notation $\models_W^V A$) if $k \models_W A$ for all $k \in K$; and \textit{assertible in $W$} (notation $\models_W^A A$) if $k \models_W A$ for some $k \in K$. For any class $\mathcal{W}' \subseteq \mathcal{W}$, we mean by $\models_{\mathcal{W'}}^V A$ that $\models_{W}^V A$ for all $W \in \mathcal{W}'$. Similar for assertibility.
\end{definition}
\noindent Plainly, $\models_W^V A$ iff $r \models_W A$; $\models_W^A A$ iff $\models_W^V \neg \neg A$; and $\not \models_W^A A$ iff $\models_W^V \neg A$. Negation's global nature is shared amongst the GN formulas (cf. section \ref{section: The basic semantic definitions}).
\begin{lemma}\label{lemma: GN's global negativity}
  For any $N \in \ClGN[\mathcal{L}(D)]$, $\models_W^A N$ implies $\models_W^V N$.
\end{lemma}
\begin{proof}
  Induction. 
\end{proof}

The notion of semantic consequence is defined similarly to the case of $\IQC$ (cf. \cite[pp.168]{vanDalen2013}), taking into account the two modes of quantification. For any set $\Gamma$ of formulas and any substitution $^\sigma$, we write $\Gamma^\sigma$ for $\{ B^\sigma \mid B \in \Gamma \}$. If $\Gamma$ is a finite set $\{B_1, ..., B_n\}$, we let $\bigwedge \Gamma$ be $B_1 \land \cdots \land B_n$, if $n \neq 0$; and be $\top$, if $n = 0$.
\begin{definition}[Semantic consequence]\label{definition: Semantic consequence}
  Let $\Gamma \subseteq_{\fin} \Fm[\mathcal{L}(D)]$, and $A \in \Fm[\mathcal{L}(D)]$. Suppose the free variables occurring in $\bigwedge \Gamma \land A$ not only globally form a sequence $\vec{x} = \langle x_1, ..., x_p\rangle$, and those occurring only globally form $\vec{y} = \langle y_1, ..., y_q \rangle$. $A$ is a \textit{semantic consequence of $\Gamma$ in $W$} (notation $\Gamma \models_W^V A$) if for any $k \in K$, $\vec{d} = \langle d_1, ..., d_p \rangle \in D^p$ and $\vec{d'} \in D^q$, whenever $k \models_W \bigwedge \{ E(\overline{d_i}) \mid 1 \leq i \leq p \} \land \bigwedge (\Gamma[\vec{d}/\vec{x}, \vec{d'}/\vec{y}])$, $k \models_W A[\vec{d}/\vec{x}, \vec{d'}/\vec{y}]$ holds. We use $\Gamma \models_{\mathcal{W}'}^V A$ with $\mathcal{W}' \subseteq \mathcal{W}$ the same way as $\models_{\mathcal{W}'}^V A$.
\end{definition}
\noindent We note that if only closed formulas are involved, $\Gamma \models_W^V A$ reduces to that for all $k$, $k \models \bigwedge \Gamma$ implies $k \models A$.
The following hold for closed formulas.
\begin{lemma}\label{lemma: several famous validities}
  (i) $\models_{\mathcal{W}}^V \wneg \wneg A \to A$. (ii) $\models_{\mathcal{W}}^V ((A \to B) \to A) \to A$. (iii) $\models_{\mathcal{W}}^V \wneg \wneg (A \lor \wneg A)$. (iv) $\models_{\mathcal{W}}^V \neg A \lor \neg \neg A$. (v) $\models_{\mathcal{W}}^V \forall x E(x)$. (vi) $\forall x \wneg \wneg A \models_{\mathcal{W}}^V \forall x A$. (vii) $\models_{\mathcal{W}}^V \forall x \wneg \wneg A \to \wneg \wneg \forall x A$.
\end{lemma}
\begin{proof}
  (i-vi) Easy. (vii) Follows from (vi). 
\end{proof}
\noindent We note that in particular, (iv) and (vii) fail in intuitionistic logic. We submit (iv) as a formalisation of the actual weak decidability (cf. sections \ref{section: The conceptual constraints}, \ref{section: Our rendition and its supports}). Also, (vii) would amount to the `double-negation shift' for intuitionistic negation.

On the other hand, the law of excluded middle and Modus Ponens fail.
\begin{lemma}
  (i) $\not \models_{\mathcal{W}}^V A \lor \wneg A$. (ii) $\not \models_{\mathcal{W}}^V A \lor \neg A$. (iii) $\models_{\mathcal{W}}^V B \to A$ and $\models_{\mathcal{W}}^V B$ do not imply $\models_{\mathcal{W}}^V A$.
\end{lemma}
\begin{proof}
  (i-iii) Let $W \in \mathcal{W}$ comprise only two nodes, $r$ and $k$ with $r < k$, and $k \models_W E(c)$ be the only forcing relation for atoms. Let $A$ be $E(c)$, and $B$ be $E(c) \to E(c)$, and look at $r$.
\end{proof}
\noindent But the premises of (iii) do imply $\models_{\mathcal{W}}^V \wneg \wneg A$: the following holds in general.
\begin{lemma}
  $(B \to A) \land B \models_{\mathcal{W}}^V \wneg \wneg A$.
\end{lemma}
\begin{proof}
  Easy.
\end{proof}
\noindent This motivates us to use the notion of `stability': $A$ is \textit{stable in $W$} if $\wneg \wneg A \models_W^V A$. Modus Ponens holds, if $A$ is stable in all $W \in \mathcal{W}$. We can provide a class with stability. Define class $\ST[\mathcal{L}]$ by
\begin{itemize}
  \item $\ST[\mathcal{L}] ::= \top \mid N \mid (S \land S) \mid (S \lor N) \mid (N \lor S) \mid (A \to A) \mid (\forall x A)$,
\end{itemize}
\indent where $N \in \GN[\mathcal{L}]$ and $A \in \Fm[\mathcal{L}]$.
\begin{lemma}
  Any $\ST$ formula is stable in all $W \in \mathcal{W}$.
\end{lemma}
\begin{proof}
  Induction.
\end{proof}
\noindent So, in fact, Modus Ponens significantly widely applies\footnote{
In section \ref{section: A proof system NSF}, we suggest that the failure of Modus Ponens would have no effect on our argument for induction in section \ref{section: Induction in strict finitism}.
}.

\subsection{A natural deduction system \textbf{NSF}}\label{section: A proof system NSF}
\textbf{NSF} has ($\land$ I), ($\land$ E), ($\lor$ I) and ($\lor$ E) of classical logic. Let in what follows $t, t_1, ..., t_n \in \Term$, $f \in \Func$, $P \in \Pred$, $1 \leq i \leq n$, $S \in \ST$ and $N \in \GN$.

\begin{multicols}{2}
  \begin{prooftree}
    \AxiomC{ \, }
    \RightLabel{$\top$ I}
    \UnaryInfC{$\top$}
  \end{prooftree}
  \begin{prooftree}
    \AxiomC{$\bot$}
    \RightLabel{$\bot$ E}
    \UnaryInfC{$A$}
  \end{prooftree}
\end{multicols}

\begin{multicols}{3}
  \begin{prooftree}
    \AxiomC{$P(t_1, ..., t_n)$}
    \RightLabel{STR$_1$}
    \UnaryInfC{$E(t_i)$}
  \end{prooftree}
  \begin{prooftree}
    \AxiomC{$E(f(t))$}
    \RightLabel{STR$_2$}
    \UnaryInfC{$E(t)$}
  \end{prooftree}
  \begin{prooftree}
    \AxiomC{$\wneg \wneg S$}
    \RightLabel{ST}
    \UnaryInfC{$S$}
  \end{prooftree}
\end{multicols}

\begin{multicols}{2}
  \begin{prooftree}
    \AxiomC{$[A]$}
    \noLine
    \UnaryInfC{$\vdots$}
    \noLine
    \UnaryInfC{$\wneg \wneg B$}
    \RightLabel{$\to$ I}
    \UnaryInfC{$A \to B$}
  \end{prooftree}
  \begin{prooftree}
    \AxiomC{$ \, $}
    \noLine
    \UnaryInfC{$ \, $}
    \noLine
    \UnaryInfC{$ \, $}
    \noLine
    \UnaryInfC{$ \, $}
    \noLine
    \UnaryInfC{$ \, $}
    \noLine
    \UnaryInfC{$ \, $}
    \noLine
    \UnaryInfC{$ \, $}
    \noLine
    \UnaryInfC{$A \to B$}
    \AxiomC{$A$}
    \RightLabel{$\to$ E}
    \BinaryInfC{$\wneg \wneg B$}
  \end{prooftree}
\end{multicols}

\begin{multicols}{2}
  \begin{prooftree}
    \AxiomC{$ \, $}
    \noLine
    \UnaryInfC{$ \, $}
    \noLine
    \UnaryInfC{$ \, $}
    \noLine
    \UnaryInfC{$ \, $}
    \noLine
    \UnaryInfC{$ \, $}
    \noLine
    \UnaryInfC{$ \, $}
    \noLine
    \UnaryInfC{$ \, $}
    \noLine
    \UnaryInfC{$\wneg N$}
    \RightLabel{$\neg$ I$_1$}
    \UnaryInfC{$\neg N$}
  \end{prooftree}
  \begin{prooftree}
    \AxiomC{$\Gamma \subseteq_{\fin} \GN$}
    \noLine
    \UnaryInfC{$\vdots$}
    \noLine
    \UnaryInfC{$\wneg A$}
    \RightLabel{$\neg$ I$_2$}
    \UnaryInfC{$\neg A$}
  \end{prooftree}
\end{multicols}

\begin{multicols}{2}
  \begin{prooftree}
    \AxiomC{$\wneg A$}
    \AxiomC{$A$}
    \RightLabel{$\to \bot$ E}
    \BinaryInfC{$\bot$}
  \end{prooftree}
  \begin{prooftree}
    \AxiomC{$\neg A$}
    \AxiomC{$A$}
    \RightLabel{$\neg$ E}
    \BinaryInfC{$\bot$}
  \end{prooftree}
\end{multicols}

\begin{multicols}{2}
  \begin{prooftree}
    \AxiomC{$ \, $}
    \noLine
    \UnaryInfC{$ \, $}
    \noLine
    \UnaryInfC{$ \, $}
    \noLine
    \UnaryInfC{$ \, $}
    \noLine
    \UnaryInfC{$ \, $}
    \noLine
    \UnaryInfC{$ \, $}
    \noLine
    \UnaryInfC{$ \, $}
    \noLine
    \UnaryInfC{$\wneg \wneg A(x)$}
    \RightLabel{$\forall$-glo I}
    \UnaryInfC{$\forall x A$}
  \end{prooftree}
  \begin{prooftree}
    \AxiomC{$[E(x)]$}
    \noLine
    \UnaryInfC{$\vdots$}
    \noLine
    \UnaryInfC{$\wneg \wneg A(x)$}
    \RightLabel{$\forall$-loc I}
    \UnaryInfC{$\forall x A$}
  \end{prooftree}
\end{multicols}

\begin{multicols}{2}
  \begin{prooftree}
    \AxiomC{$\forall x A$}
    \RightLabel{$\forall$-glo E}
    \UnaryInfC{$\wneg \wneg A[t/x]$}
  \end{prooftree}
  \begin{prooftree}
    \AxiomC{$\forall x A$}
    \AxiomC{$E(t)$}
    \RightLabel{$\forall$-loc E}
    \BinaryInfC{$\wneg \wneg A[t/x]$}
  \end{prooftree}
\end{multicols}

\begin{multicols}{2}
  \begin{prooftree}
    \AxiomC{$A[t/x]$}
    \RightLabel{$\exists$-glo I}
    \UnaryInfC{$\exists x A$}
  \end{prooftree}
  \begin{prooftree}
    \AxiomC{$A[t/x]$}
    \AxiomC{$E(t)$}
    \RightLabel{$\exists$-loc I}
    \BinaryInfC{$\exists x A$}
  \end{prooftree}
\end{multicols}

\begin{multicols}{2}
  \begin{prooftree}
    \AxiomC{$\exists y A$}
    \AxiomC{$[A[x/y]]$}
    \noLine
    \UnaryInfC{$\vdots$}
    \noLine
    \UnaryInfC{$C$}
    \RightLabel{$\exists$-glo E}
    \BinaryInfC{$C$}
  \end{prooftree}
  \begin{prooftree}
    \AxiomC{$\exists y A$}
    \AxiomC{$[A[x/y]] \, [E(x)]$}
    \noLine
    \UnaryInfC{$\vdots$}
    \noLine
    \UnaryInfC{$C$}
    \RightLabel{$\exists$-loc E}
    \BinaryInfC{$C$}
  \end{prooftree}
\end{multicols}
\noindent ($\neg$ I$_2$) is applicable if the derivation of $\wneg A$ depends only on a (possibly empty) finite set $\Gamma$ of $\GN$ formulas as open hypotheses. We note that it is not applicable, if the derivation depends on $\wneg A$ itself (i.e. $\wneg A \in \Gamma$) and $A \notin \GN$. (If $\wneg A \in \Gamma$ and $A \in \GN$, use ($\neg$ I$_1$) to conclude $\neg A$.) The `glo' rules are applicable if the variable is occurring only globally; the `loc' rules are for the other cases. $t$ in the ($\exists$ I) and ($\forall$ E) rules (glo and loc both) can be any term free for $x$. In ($\forall$ I), $x$ cannot occur free in the hypotheses on which the derivation of $A(x)$ depends (except in $\{ E(x) \}$ of ($\forall$-loc I)); in ($\exists$ E), not in $C$ or those on which the derivation of $C$ depends, except in $\{ A[x/y] \}$ of ($\exists$-glo E) and $\{ A[x/y], E(x)\}$ of ($\exists$-loc E).

For any $\Gamma \subseteq \Fm[\mathcal{L}]$ and $A \in \Fm[\mathcal{L}]$, we write $\Gamma \vdash_{\NSF} A$ for that there is a derivation of $A$ from a $\Gamma' \subseteq_{\fin} \Gamma$ in \textbf{NSF}.

\begin{proposition}[Soundness]\label{proposition: Soundness}
  If $\Gamma$ is finite and $\Gamma \vdash_{\NSF} A$, then $\Gamma \models_{\mathcal{W}}^V A$.
\end{proposition}
\begin{proof}
  Induction. Let $W = \langle K, \leq, D, J, v \rangle \in \mathcal{W}$. For brevity, we ignore insignificant substitutions. For ($\to$ I) and ($\to$ E), use that $k \models \wneg \wneg B$ iff for all $k' \geq k$, there is a $k'' \geq k'$ such that $k'' \models B$.
  
  ($\neg$ I$_1$) Suppose $k \models B$ for all $B \in \Gamma$, and $k \models \wneg N$. If $l \models N$ for some $l$, then $k \models N$ by lemma \ref{lemma: GN's global negativity}. This is contradictory. Therefore $l \not \models N$ for any $l$.
  
  ($\neg$ I$_2$) Suppose the derivation of $\wneg A$ depends only on $\Gamma \subseteq_{\fin} \GN$. Then the hypothesis is that $\Gamma \models_{\mathcal{W}}^V \wneg A$. So, if $\Gamma = \emptyset$, then $\Gamma \models_{\mathcal{W}}^V \neg A$, and we are done. So we can assume $\Gamma \neq \emptyset$. Now, suppose $k \models \bigwedge \Gamma$, and $l \models A$ for some $l$. Then $l \models \bigwedge \Gamma$ by lemma \ref{lemma: GN's global negativity}, and hence $l \models \wneg A$ by the hypothesis. This is contradictory. Therefore $l \not \models A$ for any $l$, implying $k \models \neg A$.
  
  ($\forall$-loc I) Let the free variables occurring but not only globally in $\bigwedge \Gamma \land \forall x A$ form a sequence $\vec{x} = \langle x_1, ..., x_p \rangle$, and those occurring only globally form a $\vec{y} = \langle y_1, ..., y_q \rangle$. Suppose $\vec{d} = \langle d_1, ..., d_p \rangle \in D^p$ and $\vec{d'} \in D^q$, and $k \models \bigwedge E(d_i) \land \bigwedge (\Gamma [\vec{d}/\vec{x}, \vec{d'}/\vec{y}])$. To prove $k \models \forall x (A [\vec{d}/\vec{x}, \vec{d'}/\vec{y}])$, suppose $d_x \in D$, $k' \geq k$ and $k' \models E(\overline{d_x})$. Then there is a $k'' \geq k'$ such that $k'' \models A[\vec{d}/\vec{x}, \vec{d'}/\vec{y}, \overline{d_x}/x]$, since by persistence and the hypothesis, $k' \models \wneg \wneg A[\vec{d}/\vec{x}, \vec{d'}/\vec{y}, \overline{d_x}/x]$.
  
  ($\forall$-loc E) Similar to ($\to$ E). 
\end{proof}
\noindent We note that $\{ \wneg P, \neg \neg P\} \not \vdash_{\NSF} \neg P$, if $P \in \ClAtom$. We use this as an example in section \ref{section: The proof of completeness}.

One can confirm, just as in the classical case (cf. \cite[p.90]{vanDalen1986}), that if $\Gamma \vdash_{\NSF} A$, then $\Gamma[x/c] \vdash_{\NSF} A[x/c]$, for any $x \in \Var$ that does not occur in $\Gamma \cup \{ A \}$ and for any $c \in \ClTerm$.
We note the following facts about the derivability. We let $A, B \in \Fm$, $S \in \ST$ and $N \in \GN$.
\begin{lemma}
  If $A \vdash_{\NSF} B$, then $\wneg B \vdash_{\NSF} \wneg A$ and $\neg B \vdash_{\NSF} \neg A$.
\end{lemma}
\begin{proof}
  Easy.
\end{proof}

\begin{lemma}\label{lemma: Basic facts of NSF}
  (i) $\wneg N \dashv \vdash_{\NSF} \neg N$. (ii) $\vdash_{\NSF} \wneg A$ iff $\vdash_{\NSF} \neg A$. (iii) $\vdash_{\NSF} \wneg \wneg (A \lor \wneg A)$. (iv) $\vdash_{\NSF} \neg A \lor \neg \neg A$. (v) $\vdash_{\NSF} N \lor \neg N$. (vi) $\neg \neg N \vdash_{\NSF} N$. (vii) $\exists x \neg N \vdash_{\NSF} \neg \forall x N$. (viii) $\neg \neg \exists x A \vdash_{\NSF} \exists x \neg \neg A$.
\end{lemma}
\begin{proof}
  (i, ii) Easy. (iii, iv) Similar to $A \lor \neg A$'s derivation in classical logic (cf. \cite[p.49]{vanDalen1986}). (v) Follows from (iii). (vi) Follows from (v). (vii) Apply ($\exists$-glo E) to $\exists x \neg N$. Then since the supposition of $\forall x N$ implies $\wneg \wneg N[y/x]$ for any given $y$, we obtain $\exists x \neg N \vdash_{\NSF} \wneg \forall x N$. (viii) Confirm $\exists x A \vdash_{\NSF} \exists x \neg \neg A$. Then by contraposition, $\neg \neg \exists x A \vdash_{\NSF} \neg \neg \exists x \neg \neg A$. The last proves $\exists x \neg \neg A$ by (vi).
\end{proof}

$\blacksquare$ \textit{A remark on induction} \quad We discussed in section \ref{section: Induction in strict finitism} how a strict finitist might justify the induction principle on the conceptual level. There we assumed Modus Ponens in the strict finitistic framework, but it semantically fails in our reconstruction (cf. section \ref{section: Basic properties and the semantic notions}), and the statements involved there would not translate to ST formulas. However, we suggest that this would not hinder our argument for induction. Since $\NSF$ is sound, and complete as we see in section \ref{section: Completeness}, it may be reasonable to view it as a classical counterpart of the system of strict finitistic reasoning. Confirm that Modus Ponens under $\wneg \wneg$ holds: i.e., $\wneg \wneg A, \wneg \wneg (A \to B) \vdash_{\NSF} \wneg \wneg B$. So we can recreate the semi-formal representations of the informal derivations (apart from the details of the language).

{\small 
\begin{prooftree}
  \AxiomC{$\wneg \wneg Q(2)$}
  \AxiomC{$\forall x (Q(x) \to Q(x+1))$}
  \UnaryInfC{$\wneg \wneg (Q(2) \to Q(3))$}
  \BinaryInfC{$\wneg \wneg Q(3)$}
  \AxiomC{$\forall x (Q(x) \to Q(x+1)$}
  \UnaryInfC{$\wneg \wneg (Q(3) \to Q(4))$}
  \BinaryInfC{$\wneg \wneg Q(4)$}
  \noLine
  \UnaryInfC{$\vdots$}
  \noLine
  \UnaryInfC{$\wneg \wneg Q(n+2)$.}
\end{prooftree}
}
\noindent If $\wneg \wneg Q(n+2)$, and hence $\wneg \wneg Q(n)$ is obtained for any $n$, then we suggest that one can conclude that $\forall x Q(x)$ by an informal analogue of ($\forall$-glo I). Surely, the derivation of Modus Ponens under $\wneg \wneg$ takes $4$ steps, and therefore the entire derivation is longer than in \ref{section: Induction in strict finitism}. But this concern may be circumvented, by fine-tuning the details, or at least by adopting 
\begin{prooftree}
  \AxiomC{$\wneg \wneg A$}
  \AxiomC{$\wneg \wneg (A \to B)$}
  \BinaryInfC{$\wneg \wneg B$}
\end{prooftree}
as a rule.

\subsection{Completeness}\label{section: Completeness}
In this section, we prove the completeness of $\NSF$ (section \ref{section: A proof system NSF}) with respect to the strict finitistic semantics (section \ref{section: Semantics}): i.e., for any $\Gamma \subseteq_{\fin} \ClForm[\mathcal{L}]$ and $A \in \ClForm[\mathcal{L}]$, if $\Gamma \models_{\mathcal{W}}^V A$, then $\Gamma \vdash_{\NSF}A$. We will prove its contraposition, and our method does not greatly differ from the case of $\IQC$: we start with $\Gamma \not \vdash_{\NSF} A$, and construct a countermodel to $\Gamma \models_{\mathcal{W}}^V A$ by extending $\Gamma$ to `prime theories' and forming the frame of the countermodel. However, we will need some significant modifications in doing so. In what follows, we first develop our ways of extending a theory (section \ref{section: Preliminaries}), and then proceed to the proof of the completeness itself (section \ref{section: The proof of completeness}).

\subsubsection{Preliminaries}
\label{section: Preliminaries}

We will see that we can extend a `consistent' set of formulas to a theory similarly to the case of $\IQC$ (lemma \ref{lemma: Extension lemma}), and make a useful observation about extending a theory with a set of GN formulas (lemma \ref{lemma: GN extension lemma}).

While $\mathcal{L}$ is our fixed language, we will define several notions in a more general setting. Let $\mathcal{L}_0$ be any language, and $\Sigma \subseteq \ClForm[\mathcal{L}_0]$. We say $\Sigma$ is \textit{consistent with respect to $\NSF$} if $\Sigma \not \vdash_{\NSF} \bot$. 
\begin{definition}[Prime theory]\label{definition: Prime theory}
  $\Sigma$ is a \textit{prime theory in $\mathcal{L}_0$} if for any $A, B, \exists x C \in \ClForm[\mathcal{L}_0]$, (i) $\Sigma \vdash_\textbf{NSF} A$ implies $A \in \Sigma$, (ii) $A \lor B \in \Sigma$ implies $A \in \Sigma$ or $B \in \Sigma$ and (iii) if $\exists x C \in \Sigma$, then there is a constant $a$ in $\mathcal{L}_0$ such that (iii-a) $C[a/x] \in \Sigma$ if $x$ occurs only globally, or (iii-b) $E(a), C[a/x] \in \Sigma$ otherwise.
\end{definition}
\noindent Also, we call the \textit{$\Sigma$-fragment of $\mathcal{L}_0$}, the language $\mathcal{L}_0$ whose predicate symbols are only those occurring in $\Sigma$. We write $\mathcal{L}_{_0 \vert \Sigma}$ for it.

Now, fix a $\Gamma \subseteq_{\fin} \ClForm[\mathcal{L}]$, and let $\mathcal{L}' = \mathcal{L}(M)$ for some countable set $M$ of constants.
\begin{lemma}[Extension lemma]\label{lemma: Extension lemma}
  Let $A \in \ClForm[\mathcal{L}]$ be such that $\Gamma \not \vdash_{\NSF} A$. Then there is a prime theory $\Gamma'$ in $\mathcal{L}'_{\vert \Gamma}$ such that $\Gamma \subseteq \Gamma'$ and $\Gamma' \not \vdash_{\NSF} A$.
\end{lemma}
\begin{proof}
  We will inductively construct a sequence $\langle \Gamma_n \rangle_{n \in \mathbb{N}}$. Let $\langle B_p \lor C_p \rangle_{p \in \mathbb{N}}$ enumerate with infinite repetition all disjunctive formulas in $\ClForm[\mathcal{L}'_{\vert \Gamma}]$; $\langle (\exists x B)_p \rangle_{p \in \mathbb{N}}$ all existential formulas in $\ClForm[\mathcal{L}'_{\vert \Gamma}]$ with $x$ occurring only globally; and $\langle (\exists x C)_p \rangle_{p \in \mathbb{N}}$ the same, except with $x$ not occurring only globally (possibly $x$ not even occurring). Define $\Gamma_0 := \Gamma$. Assume $\Gamma_n$ with $\Gamma_n \not \vdash_{\NSF} A$ is given.
  \begin{enumerate}
    \item If $n = 3p$ for some $p$, consider $B_p \lor C_p$. If $\Gamma_n \not \vdash_{\NSF} B_p \lor C_p$, then define $\Gamma_{n+1} := \Gamma_n$. If $\Gamma_n \vdash_{\NSF} B_p \lor C_p$, confirm that $\Gamma_n \cup \{ B_p \} \not \vdash_{\NSF} A$ or $\Gamma_n \cup \{ C_p \} \not \vdash_{\NSF} A$. Define $\Gamma_{n+1} := \Gamma_n \cup \{ B_p \}$ if the former; $\Gamma_{n+1} := \Gamma_n \cup \{ C_p \}$ otherwise.
    \item If $n = 3p+1$ for some $p$, consider $(\exists x B)_p$. If $\Gamma_n \not \vdash_{\NSF} (\exists x B)_p$, then define $\Gamma_{n+1} := \Gamma_n$. If $\Gamma_n \vdash_{\NSF} (\exists x B)_p$, take an $a \in M$ not occurring in $\Gamma_n$. Define $\Gamma_{n+1} := \Gamma_n \cup \{ B[a/x] \}$.
    \item If $n = 3p+2$ for some $p$, consider $(\exists x C)_p$. If $\Gamma_n \not \vdash_{\NSF} (\exists x C)_p$, then define $\Gamma_{n+1} := \Gamma_n$. If $\Gamma_n \vdash_{\NSF} (\exists x C)_p$, take an $a \in M$ not occurring in $\Gamma_n$. Define $\Gamma_{n+1} := \Gamma_n \cup \{ C[a/x], E(a) \}$.
  \end{enumerate}
  Define $\Gamma' := \bigcup_{n \in \mathbb{N}} \Gamma_n$.
  
  Then obviously $\Gamma \subseteq \Gamma' \subseteq \ClForm[\mathcal{L}'_{\vert \Gamma}]$. To prove $\Gamma' \not \vdash_{\NSF} A$, let us show by induction that $\Gamma_n \not \vdash_{\NSF} A$ for all $n$. First $\Gamma_0 \not \vdash_{\NSF} A$. Suppose $\Gamma_n \not \vdash_{\NSF} A$. If $n = 3p+2$ and $\Gamma_{n+1} \vdash_{\NSF} A$, then $\Gamma_n \cup \{ E(a), C[a/x] \} \vdash_{\NSF} A$. Then, for any $y \in \Var$ not occurring in $\Gamma_n \cup \{ E(a), C[a/x], A \}$, we have $\Gamma_n[y/a] \cup \{ E(a)[y/a], C[a/x][y/a] \} \vdash_{\NSF} A[y/a]$, i.e., $\Gamma_n \cup \{ E(y), C[y/x] \} \vdash_{\NSF} A$. So by applying ($\exists$-loc E) to $\exists x C$, $\Gamma_n \vdash_{\NSF} \exists x C$ yields a contradiction. The other cases are easier.
  
  $\Gamma'$ is a prime theory. For, first, if $B \lor C \in \Gamma'$, then $\Gamma' \vdash_{\NSF} B \lor C$. So let $n$ be the least number such that $\Gamma_n \vdash_{\NSF} B \lor C$. Then by construction there is a $p$ such that $3p \geq n$ and we consider $B \lor C$ at the $3p$-th step. Therefore $B \in \Gamma_{3p+1}$ or $C \in \Gamma_{3p+1}$. Next, if $\exists x B \in \Gamma'$, argue similarly. Finally, if $\Gamma' \vdash_\textbf{NSF} B$, then $\Gamma' \vdash_\textbf{NSF} B \lor B$. Therefore $B \in \Gamma'$.
\end{proof}
\noindent We call $\Gamma'$ thus obtained an \textit{extension of $\Gamma$ excluding $A$}.

Next, we see that to add new GN formulas in a larger language to a theory is to extend it conservatively, if the GN formulas form a consistent `prime GN theory'.
\begin{definition}[Prime GN theory]\label{definition: Prime GN theory}
  A $\Lambda \subseteq \ClGN[\mathcal{L}_0] $ is a \textit{prime GN theory in $\mathcal{L}_0$} if for any $N_1, N_2, \exists x N_3 \in \ClGN[\mathcal{L}_0]$, (i) $\Lambda \vdash_{\NSF} N_1$ implies $N_1 \in \Lambda$, (ii) $N_1 \lor N_2 \in \Lambda$ implies $N_1 \in \Lambda$ or $N_2 \in \Lambda$ and (iii) if $\exists x N_3 \in \Lambda$, then there is a constant $a$ in $\mathcal{L}_0$ such that $N_3[a/x] \in \Lambda$.
\end{definition}
\noindent Thus a prime GN theory is a prime theory only with respect to GN formulas.

Let us write $\Sigma_{\GN}$ for $\Sigma \cap \ClGN[\mathcal{L}_0]$ for any $\Sigma \subseteq \ClForm[\mathcal{L}_0]$. Then, for any consistent prime theory $\Gamma'$ in $\mathcal{L}'_{\vert \Gamma}$, $\Gamma'_{\GN}$ is a prime GN theory in $\mathcal{L}'_{\vert \Gamma}$. For, ($\lor$) if $B \lor C \in \Gamma'_{\GN} \subseteq \Gamma'$, then $B \in \Gamma'$ or $C \in \Gamma'$; and since $B, C \in \ClGN[\mathcal{L}'_{\vert \Gamma}]$, we have $B \in \Gamma'_{\GN}$ or $C \in \Gamma'_{\GN}$.

Now, fix a further extended $\mathcal{L}^* = \mathcal{L}'(M')$, for some countable $M'$.
\begin{lemma}[GN extension lemma]\label{lemma: GN extension lemma}
  Let $\Lambda$ be a consistent prime GN theory in $\mathcal{L}^*_{\vert \Gamma}$ such that $\Lambda \cap \ClForm[\mathcal{L}'_{\vert \Gamma}] \subseteq \Gamma'_{\GN}$. Then, for any $A \in \ClForm[\mathcal{L}'_{\vert \Gamma}]$, $\Gamma' \cup \Lambda \vdash_{\NSF} A$ implies $\Gamma' \vdash_{\NSF} A$.
\end{lemma}
\begin{proof}
  Prove that for all $n \in \mathbb{N}$, $\Gamma'_0 \subseteq_{\fin} \Gamma'$ and $N_1, ..., N_n \in \Lambda$, if there is a derivation from $\Gamma'_0 \cup \{ N_1, ..., N_n \}$ to $A$, then $\Gamma' \vdash_{\NSF} A$.
  
  If $n$ is $0$, we are done. So suppose $n=1$ and $\Gamma'_0 \cup \{ N_1 \} \vdash_{\NSF} A$. If $N_1 \in \ClGN[\mathcal{L}'_{\vert \Gamma}]$, then $N_1 \in \Gamma'_{\GN} \subseteq \Gamma'$.
  
  So we can assume $N_1 \notin \ClGN[\mathcal{L}'_{\vert \Gamma}]$. Suppose that only one constant $a \in M'$ occurs in $N_1$. Then we have $\Gamma_0' \cup \{ N_1[y/a] \} \vdash_{\NSF} A$ for any $y \in \Var$ not occurring in $\Gamma'_0 \cup \{ N_1, A \}$, since $\Gamma' \subseteq \ClForm[\mathcal{L}'_{\vert \Gamma}]$. So it follows that $\Gamma'_0 \cup \{ \exists x \, (N_1[x/a]) \} \vdash_{\NSF} A$ by ($\exists$-glo E). Here, $\exists x \, (N_1[x/a]) \in \Lambda \cap \ClGN[\mathcal{L}'_{\vert \Gamma}]$. Therefore $\exists x \, (N_1[x/a]) \in \Gamma'$, implying $\Gamma' \vdash_{\NSF} A$. The case where $N_1$ has more than one constant in $M'$ is similar.
  
  If $2 \leq n$, then $\Gamma_0 \cup \{ N_1 \land \cdots \land N_k \}$ derives $A$, since $\Gamma_0 \cup \{ N_1, ..., N_k \}$ derives $A$. Thus this case reduces to the case where $n=1$.
\end{proof}
\noindent Thus we call $\Gamma' \cup \Lambda$ a \textit{GN extension of $\Gamma'$}.

\subsubsection{The proof of completeness}
\label{section: The proof of completeness}

We denote by $\langle \mathbb{N}^{< \omega}, \sqsubseteq \rangle$ the tree of all finite sequences of the natural numbers ordered by the initial segment relation ($\sqsubseteq$); and call any tree isomorphic to it a \textit{$\nabla$-tree}. We write $^\frown$ for concatenation of sequences; and $\vert \vec{n} \vert$ for the length of a sequence $\vec{n}$. We will be discussing structures whose nodes are associated with theories. For convenience, we may identify the nodes with the theories, calling them `node-theories'.

We will prove the completeness (as proposition \ref{proposition: Strong completeness NSF}) by adapting the standard method for the case of $\IQC$. Let us briefly recall its essential part. Given $\Gamma \not \vdash_{\IQC} A$ ($\vdash_{\IQC}$ being the derivability in an intended proof system of $\IQC$), one extends $\Gamma$ to a prime theory $\Gamma'$ in an extended $\mathcal{L}' = \mathcal{L}(M)$ that excludes $A$. Then, using $\Gamma'$, one constructs a $\nabla$-tree of node-theories which becomes the frame of the canonical countermodel. To do so, one first assumes a countable sequence $\mathcal{L}_0, \mathcal{L}_1, ...$ of languages, with $\mathcal{L}_0 = \mathcal{L}'$, such that the sets of the constants are increasing in the sense that $\mathcal{L}_1 = \mathcal{L}_0(M_0)$, $\mathcal{L}_2 = \mathcal{L}_1(M_1)$, ... with mutually exclusive countable sets $M_0, M_1, ...$ of constants. Then one defines each node-theory $\vec{n} \in \mathbb{N}^{< \omega}$, inductively on $\vert \vec{n} \vert$, to be a prime theory in language $\mathcal{L}_{\vert \vec{n} \vert}$. The root $\langle \rangle$ is defined to be $\Gamma'$. Given a node-theory $k$ defined, one enumerates the pairs $\langle \langle \sigma_n, \tau_n \rangle \rangle_{n \in \mathbb{N}}$ of the closed formulas in $\mathcal{L}_{\vert k \vert + 1}$ such that $k \cup \{ \sigma_n \} \not \vdash_{\IQC} \tau_n$; and for each $n$, extends the left side to a prime theory in $\mathcal{L}_{\vert k \vert +1}$ excluding the right side, and lets it be $k^\frown \langle n \rangle$, one of $k$'s direct successors in terms of $\sqsubseteq$. With the construction done, one only needs to show that $B \in k$ iff $k$ intuitionistically forces $B$ (the `truth lemma'). Indeed, one needs the $\nabla$-tree of the prime theories for the ($\to$)- and the $(\forall)$-clauses of this lemma.

$\blacksquare$ \textit{The construction method} \quad Our basic strategy is the same as above. We saw in lemma \ref{lemma: Extension lemma} that we can extend $\Gamma$ to a prime theory $\Gamma'$ in $\mathcal{L}'_{\vert \Gamma}$ for an extended $\mathcal{L}' = \mathcal{L}(M)$. The whole point of the proof consists in our suitable way of placing the other theories to define the frame (lemma \ref{lemma: Construction of the tree of theories}) that secures our truth lemma (lemma \ref{lemma: Truth lemma}). On the one hand, we will take care of the implicational and the locally universal formulas in a node-theory $k$ by constructing a $\nabla$-tree above $k$ with some fine tuning.

On the other, negation and global universal quantification give rise to issues to consider. We will explain our method by describing them and providing how we will treat. We hope that our method can be motivated this way, but we do not claim that it is the only construction method possible.

We will discuss 2 issues on the frame's general shape; and 3 on how each node-theory should be constructed. Henceforth we will only be dealing with the $\Gamma$-fragment of languages, and omit the subscript $_{\vert \Gamma}$.

(i) The entire frame's root will be $\Gamma'_{\GN}$; $\Gamma'$ will be a direct successor of $\Gamma'_{\GN}$. Let us consider the case where $\Gamma = \{ \wneg P, \neg \neg P\}$ ($P$ atomic), and suppose $\Gamma'$ is obtained. We will let $\Gamma'$ be one node-theory, and construct a $\nabla$-tree above it, but then, for the truth lemma to hold, it is necessary that $k \not \models P$ for all $k \sqsupseteq \Gamma'$, and yet $l \models P$ for some $l$. This $l$ must be outside of the $\nabla$-tree; and this means that there must be another node-theory $k_0$ below $\Gamma'$ which branches into $\Gamma'$ and the portion of the frame $l$ belongs to. We choose $\Gamma'_{\GN}$ to be $k_0$, and this will serve as the entire frame's root.

(ii) We will have countably many $\nabla$-trees. This need comes from a similar consideration to (i). Let $k$ be a node-theory in a $\nabla$-tree. Then, since $\vdash_{\NSF} \neg B \lor \neg \neg B$ for all $B$, $\neg \neg (Q \to R) \in k$ ($Q, R$ atomic) may be the case. Also, since $\vdash_{\NSF} \wneg \wneg (B \lor \wneg B)$, $k \cup \{ (Q \to R) \lor \wneg (Q \to R) \} \not \vdash_{\NSF} \bot$. Therefore there will (according to our construction) be a direct successor $k'$ of $k$ such that $(Q \to R) \lor \wneg (Q \to R) \in k'$. So $\wneg (Q \to R) \in k'$ may be the case in addition to $\neg \neg (Q \to R) \in k'$. So, similarly to (i) above, we must prepare a node-theory containing an implicational $Q \to R$ outside of the $\nabla$-tree in question, with its own $\nabla$-tree. This process repeats countably many times, since extending a theory involves a new language. So we will place $\nabla$-trees $\nabla_0, \nabla_1, ...$ horizontally and connect them by the root $\Gamma'_{\GN}$ below, letting $\Gamma'$ be $\nabla_0$'s root.

Thus we are led to the following description of the frame. For each $i \in \mathbb{N}$, we call $r_i$ the root of the $i$-th $\nabla$-tree, i.e. $\nabla_i$; and we call the entire frame's root $r_*$. We denote each node in $\nabla_i$ by a pair $\langle i, \vec{n} \rangle$, where $\vec{n} \in \mathbb{N}^{< \omega}$: e.g., the root $r_i$ is $\langle i, \langle \rangle \rangle$. Let us write $\dot{\nabla}_i$ for $\{ \langle i, \vec{n} \rangle \mid \vec{n} \in \mathbb{N}^{< \omega }\}$, which is the underlying set of $\nabla_i$. The frame we will be constructing is $\langle K, \leq \rangle$ where
\begin{itemize}
  \item $K = \{ r_*, \bigcup_{i \in \mathbb{N}} \dot{\nabla}_i \}$,
  \item $k \leq k'$ iff (i) $k = r_*$ or (ii) $k = \langle i, \vec{m} \rangle$ and $k' = \langle i, \vec{n} \rangle$ for some $i$ and $\vec{m} \sqsubseteq \vec{n}$,
  \item $k \leq k'$ implies $k \subseteq k'$ as sets and
  \item $r_* = \Gamma'_{\GN}$ and $r_0 = \Gamma'$.
\end{itemize}

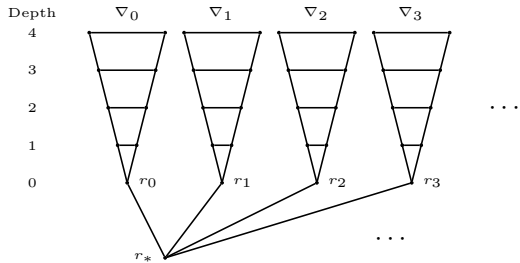
\ctikzfig{Figure1}


(iii) $r_*, r_1, r_2, ...$ will be prime GN theories; the rest will be prime theories. Suppose e.g. $\wneg P, \neg \neg P \in k$, and let us consider how we should construct a node-theory containing $P$. Defining it to be $l = k \cup \{ B \mid \neg \neg B \in k \}$ does not succeed. For $\wneg P \in k$ implies $\neg \neg \wneg P \in k$, resulting in that $P, \wneg P \in l$, a contradiction. So we will avoid incorporating all such $B$ in bulk. Instead, we let some prime GN theory $\Lambda$ with $k_{\GN} \subseteq \Lambda$ be a new root $r_i$, and simply construct a $\nabla$-tree of prime theories above it, while enumerating all pairs $\langle \sigma, \tau \rangle$ such that $\Lambda \cup \{ \sigma \} \not \vdash_{\NSF} \tau$. This way, for each $B$ with $\neg \neg B \in k$, a suitable node-theory is obtained: since $\neg \neg B \in k$ implies $\Lambda \cup \{ B \} \not \vdash_{\NSF} \bot$, $\langle B, \bot \rangle$ is picked up. We note that this entails that for any $k \in K$, some $r_i$ must be constructed later.

(iv) Each node's direct successors should be constructed in order. As we construct a $\nabla$-tree of prime theories above an $r_i$, suppose that we construct a node-theory $k$'s two direct successors $k_1$ and $k_2$ independently, just as in the case of $\IQC$. Then $\neg B \in k_1$ and $\neg \neg B \in k_2$ for some $B \in \ClForm[\mathcal{L}_{\vert k \vert +1}]$ may happen, since $\vdash_{\NSF} \neg B \lor \neg \neg B$ and $B$ is taken from a new language. To avoid this, we will inductively construct all direct successors of $k$ in a well-ordering $\preccurlyeq$\footnote{
Not to be confused with $\preceq$ for a generation order in section \ref{section: A prevalent model as an intuitionistic node}.
}.
Accordingly, each direct successor $k_n$ will be a prime theory in a distinct language $\mathcal{L}(k_n)$, as opposed to the case of $\IQC$; and the languages will be increasing in the same order $\preccurlyeq$. Also, we enumerate all pairs $\langle \sigma, \tau \rangle$ of the closed formulas in $\mathcal{L}(k)$ such that $k \cup \{ \sigma \} \not \vdash_{\NSF} \tau$. As we construct the direct successors, we carry over the set $\Lambda$ of all GN formulas so far. Given the node-theories until, say, $k_1$, we extend the left side of $k \cup \{ \sigma \} \cup \Lambda \not \vdash_{\NSF} \tau$, which is implied by $k \cup \{ \sigma \} \not \vdash_{\NSF} \tau$ by lemma \ref{lemma: GN extension lemma}, to obtain the next node-theory. This way, $\neg \neg B \in k_2$ is precluded, if $\neg B \in k_1$ and $k_1 \prec k_2$.

(v) We will construct the node-theories moving between the $\nabla$-trees back and forth. Suppose a global $\forall x B$ is in a $k \in K$. Then, for any $l \in K$, $c \in \ClTerm[\mathcal{L}(l)]$ and $k' \geq k$, there must be a $k'' \geq k'$ such that $k'' \models B(c)$, in order for the truth lemma to hold. This entails that for any $k', l \in K$, there must be a node-theory $k''$ above $k'$ whose language $\mathcal{L}(k'')$ contains all constants of $\mathcal{L}(l)$. So we will construct the node-theories in a way that, for any $i, p \in \mathbb{N}$, after constructing all nodes $\langle i, \vec{n} \rangle$ with $\vert \vec{n} \vert = p$, we revisit each $\nabla$-tree so far (i.e., each $\nabla_j$ with $j < i$) and construct new nodes for them, above the existing ones\footnote{
Negation will partly require this construction also.
}.

With these considerations, we set a well-ordering $\preccurlyeq$ on $\bigcup_{i \in \mathbb{N}} \dot{\nabla}_i$ and use it as the order of node-theory construction and language-inclusion. It must meet the following requirements.
\begin{itemize}
  \item [(1)] For each $k \in \bigcup_{i \in \mathbb{N}} \dot{\nabla}_i$, there is an $i \in \mathbb{N}$ such that $k \preccurlyeq r_i$ (for iii).
  \item [(2)] $k \preccurlyeq l$ implies $k_{\GN} \subseteq l_{\GN}$ (for iv).
  \item [(3)] For any $k, l$ and $k' \geq k$, there is a $k'' \geq k'$ such that $l \preccurlyeq k''$ (for v).
\end{itemize}
As a concrete relation, we define $\preccurlyeq$ by combining the lexicographic order (`lex', henceforth) with $\leq_{\mathbb{N}^2}$, the finite part of the canonical well-ordering on the ordinal pairs (cf. \cite[pp.30-1]{Jech2002}), as follows.
\begin{definition}[$\leq_{\mathbb{N}^2}$, $\preccurlyeq$]\label{definition: Well-ordering preceq}
  (i) Define $<_{\mathbb{N}^2} \subseteq (\mathbb{N} \times \mathbb{N})^2$ by $\langle m, n \rangle <_{\mathbb{N}^2} \langle p, q \rangle$ iff (1) $\max \{ m, n \} < \max \{ p, q \}$ or (2) $\max \{ m, n \} = \max \{ p, q \}$ and $m < p$ or (3) $\max \{ m, n \} = \max \{ p, q \}$, $m = p$ and $n < q$.
  
  (ii) Define $\prec \subseteq (\mathbb{N} \times \mathbb{N}^{< \omega})^2$ by $\langle i, \vec{m} \rangle \prec \langle j, \vec{n} \rangle$ iff (1) $\langle i, \vert \vec{m} \vert \rangle <_{\mathbb{N}^2} \langle j, \vert \vec{n} \vert \rangle$ or (2) $\langle i, \vert \vec{m} \vert \rangle = \langle j, \vert \vec{n} \vert \rangle$ and $\vec{m}$ is less than $\vec{n}$ with respect to lex.
\end{definition}

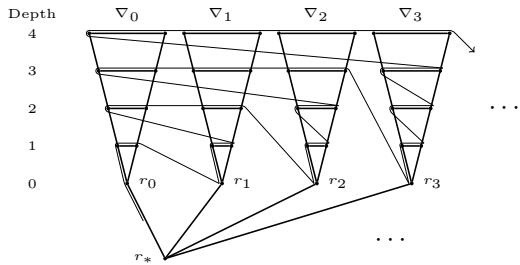
\ctikzfig{Figure2}


\noindent $\preccurlyeq$ thus defined is a well-ordering, since lex is a well-ordering on the finite sequences with the same length. We denote $\langle m, n \rangle$'s direct successor in terms of $\leq_{\mathbb{N}^2}$ by $S_{\mathbb{N}^2}(m, n)$; and $\langle i, \vec{m} \rangle$'s direct successor in terms of $\preccurlyeq$ by $S_{\preccurlyeq}(i, \vec{m})$.

If $\preccurlyeq$ is considered in the context of a tree $\langle \{ r_*, \bigcup_{i \in \mathbb{N}} \dot{\nabla}_i \}, \leq \rangle$, then $\vec{m} \sqsubseteq \vec{n}$ implies $\langle i, \vec{m} \rangle \preccurlyeq \langle i, \vec{n} \rangle$ for all $i$, and requirements (1) and (3) above are satisfied. So we only need to construct our frame $\langle K, \leq \rangle$, with $r_*, r_1, r_2, ...$ being prime GN theories and the rest prime theories, so that (2) is satisfied.

With this ordering $\preccurlyeq$, each node-theory $r_{i+1}$ ($i \in \mathbb{N}$) will be defined to be $\bigcup_{k \prec r_{i+1}} k_{\GN}$, the set of all $\GN$ formulas so far\footnote{
This is not circular. `$k \prec r_{i+1}$' refers to $r_{i+1}$ as a node, not as a set of formulas.
}.
We will set $\mathcal{L}(r_{i+1})$ to be the union of all languages so far. The following guarantees that $r_{i+1}$ will then be a consistent prime GN theory, if the hypotheses of the induction are given.
\begin{lemma}\label{lemma: ri is a prime GN theory}
  Suppose each $k \prec r_{i+1}$ is a consistent prime theory or a consistent prime $\GN$ theory in $\mathcal{L}(k)$; and for all $k, k' \prec r_{i+1}$, if $k \preccurlyeq k'$ then $k_{\GN} \subseteq k'_{\GN}$. Then $\bigcup_{k \prec r_{i+1}} k_{\GN}$ is (i) consistent, and (ii) a prime $\GN$ theory in $\mathcal{L}(r_{i+1})$.
\end{lemma}
\begin{proof}
  (i) If $\bigcup_{k \prec r_{i+1}} k_{\GN} \vdash_{\NSF} \bot$, then since $\preccurlyeq$ is a well-ordering, there is a least $k \prec r_{i+1}$ such that $k_{\GN} \vdash_{\NSF} \bot$. This is contradictory. (ii) Similar.
\end{proof}

We note that a minor modification is also needed for the enumeration process in constructing the direct successors. As described in (iv), we enumerate $\langle \sigma, \tau \rangle$ in $\ClForm[\mathcal{L}(k)]$, not in the language of a direct predecessor. But to cover the universal formulas, we need to involve a constant in a direct successor's language. So we will also enumerate all $\forall x B \in \ClForm[\mathcal{L}(k)]$ such that $k \not \vdash_{\NSF} \forall x B$, and alternatingly with the pairs $\langle \sigma, \tau \rangle$, we extend the left side of $k \cup \{ E(a) \land \wneg B(a) \} \not \vdash_{\NSF} \bot$, with $a$ taken from a direct successor's language.

$\blacksquare$ \textit{The construction of the frame} \quad If $k$ is $\langle i, \vec{n} \rangle$, then we write $k^\frown \vec{m}$ to denote $\langle i, \vec{n}^\frown \vec{m} \rangle$. For simplicity, we often omit `consistent'. 

\begin{lemma}[Construction of the tree of theories]\label{lemma: Construction of the tree of theories}
  There is a tree ${\langle K, \leq \rangle}$ of theories described above for which requirement (2) holds.
\end{lemma}
\begin{proof}
  We will construct $K = \{ r_*, \bigcup_{i \in \mathbb{N}} \dot{\nabla}_i \}$ by induction in the order of $\preccurlyeq$ (definition \ref{definition: Well-ordering preceq}), after defining $r_*$ and $r_0$. $r_0$ is the initial element of the induction. We will associate (and identify for simplicity) every node $k \in K$ with a prime theory or a prime GN theory in the language $\mathcal{L}(k)$ defined below.
  
  Consider $\mathcal{L}'(M^\nabla)$ for some countable $M^\nabla$ that is disjoint with $\mathcal{L}'$'s constants. We assume $M^\nabla$ is partitioned by $K$'s non-root elements $K \backslash \{ r_*, r_0, r_1, ... \}$ in the sense that a set $M^\nabla(k)$ is defined for each $k \in K$ so that
  \begin{itemize}
    \item $M^\nabla(r_*) = M^\nabla(r_i) = \emptyset$ for each $i \in \mathbb{N}$,
    \item $M^\nabla(i, \vec{n})$ is countable for each $i$ and nonempty $\vec{n}$,
    \item $M^\nabla = \bigcup_{k \in K} M^\nabla(k)$, and
    \item $M^\nabla(i, \vec{m})$ and $M^\nabla(j, \vec{n})$ are disjoint, if $\vec{m}$ and $\vec{n}$ are nonempty.
  \end{itemize}
  For each $k \in K$, we define $\mathcal{L}(k)$ to be $\mathcal{L}'(\bigcup_{k' \preccurlyeq k} M^\nabla(k'))$.
  
  Now let us construct the theories. We will follow the order $\leq_{\mathbb{N}^2}$ in the larger scale. In the smaller scale, we construct the nodes $\langle i, \vec{n} \rangle$ with a fixed $\vert \vec{n} \vert$ following lex.
  
  \begin{enumerate}
    \item Let $r_* = \Gamma'_{\GN}$ and $r_0 = \Gamma'$ as described.
    \item Suppose for some $i, p \in \mathbb{N}$, all the nodes $\langle i, \vec{n} \rangle$ with $\vert \vec{n} \vert = p$ are defined to be prime theories or prime GN theories such that for all $k, k'$ so far defined, $k \preccurlyeq k'$ implies $k_{\GN} \subseteq k'_{\GN}$. We will construct the nodes $\langle j, \vec{m} \rangle$ with $\langle j, \vert \vec{m} \vert \rangle = S_{\mathbb{N}^2}(i, p)$. For any node $k \in K$, if all $k' \prec k$ are defined (as sets of formulas), then we write $\Lambda(k)$ for $\bigcup_{k' \prec k} k'_{\GN}$.
    \begin{enumerate}
      \item If $\vert \vec{m} \vert = 0$, then $j = i+1$. Define the only one node-theory $\langle i+1, \langle \rangle \rangle$, i.e. $r_{i+1}$, to be $\Lambda(r_{i+1})$. This is a prime GN theory in $\mathcal{L}(r_{i+1})$ by lemma \ref{lemma: ri is a prime GN theory}.
      \item Suppose $\vert \vec{m} \vert \neq 0$, and let $k$ denote $\langle j, \langle 0, ..., 0 \rangle \rangle$, where the length of $\langle 0, ..., 0 \rangle$ is $\vert \vec{m} \vert - 1$. We note that $k^\frown \langle 0 \rangle$ is the least of $\{ \langle j, \vec{m} \rangle \mid \langle j, \vert \vec{m} \vert \rangle = S_{\mathbb{N}^2}(i, p) \}$ according to lex. The node-theory $k$ is defined already, since $k \sqsubset k^\frown \langle 0 \rangle$. We consider two sequences.
      \begin{enumerate}
	\item Let $\langle \langle \sigma_q, \tau_q \rangle \rangle_{q \in \mathbb{N}}$ enumerate all the pairs $\langle \sigma, \tau \rangle$ in $\ClForm[\mathcal{L}(k)]$ such that $k \cup \{ \sigma \} \not \vdash_{\NSF} \tau$. 
	\item Let $\langle (\forall x B)_q \rangle_{q \in \mathbb{N}}$ enumerate all the universal $\forall x B \in \ClForm[\mathcal{L}(k)]$ such that $k \not \vdash_{\NSF} \forall x B$.
      \end{enumerate}
      Inductively construct $k$'s successors $k^\frown \langle n \rangle$ to be prime theories in $\mathcal{L}(k^\frown \langle n \rangle)$ ($n \in \mathbb{N}$) as follows.
      \begin{enumerate}
	\item If $n = 2q$ for some $q \in \mathbb{N}$, then confirm that $k \cup \{ \sigma_q \} \cup \Lambda(k^\frown \langle n \rangle) \not \vdash_{\NSF} \tau_q$. For, $\Lambda(k^\frown \langle n \rangle)$ is a prime GN theory in $\mathcal{L}( \bigcup_{l \prec \langle k^\frown \langle n \rangle \rangle} l)$; and hence the left side is a GN extension of $k \cup \{ \sigma_q \}$. So extend the left side excluding $\tau_{q}$ to obtain $k^\frown \langle n \rangle$.
	\item If $n = 2q+1$ for some $q \in \mathbb{N}$, then take an $a \in M^\nabla(k^\frown \langle n \rangle)$.
	\begin{itemize}
	  \item If $x$ of $(\forall x B)_q$ is not occurring only globally in $B$, confirm that $k \cup \{ E(a) \land \wneg B(a)) \} \not \vdash_{\NSF} \bot$. For, otherwise we would have $k \vdash_{\NSF} (\forall x B)_q$. So extend the left side excluding $\bot$ to obtain $k^\frown \langle n \rangle$.
	  \item If $x$ is occurring only globally, do the same dropping $E(a)$.
	\end{itemize}
      \end{enumerate}
      Given $k^\frown \langle n \rangle$ defined for all $n$, consider $S_{\preccurlyeq}(k)$ and do the same to construct its direct successors $(S_{\preccurlyeq}(k))^\frown \langle n \rangle$ for all $n$. Repeat this process to construct all $\langle j, \vec{m} \rangle$ with $\langle j, \vert \vec{m} \vert \rangle = S_{\mathbb{N}^2}(i, p)$: in general, given $\langle j, \vec{m'}^\frown \vec{n'} \rangle$ defined for all $\vec{n'}$ with $\vert \vec{m'}^\frown \vec{n'} \vert = p$, proceed to constructing each $(S_{\preccurlyeq}(j, (\vec{m'})))^\frown \vec{n'}$.
    \end{enumerate}
  \end{enumerate}
  
  Then plainly requirement (2) holds by construction.
\end{proof}

$\blacksquare$ \textit{The definition of the model and the truth lemma} \quad On the tree $\langle K, \leq \rangle$ thus obtained, we define $W = {\langle K, \leq, D, J, v \rangle}$ as follows. $D := \ClTerm[\mathcal{L}'(M^\nabla)]$. Inductively, for each constant $a$ in $\mathcal{L}'(M^\nabla)$, $J^1(a) := a$; for each $d \in D$, $J^1(\overline{d}) := d$; for each $f \in \Func$, $(J^2 (f)) (d) := f(d)$ for all $d \in D$; and for each $f(c) \in \ClTerm[\mathcal{L}(D)]$, $J^1( f(c) ) := (J^2(f)) (J^1(c))$.  For any $P \in \Pred$ and $k \in K$, $\langle d \rangle \in P^{v(k)}$ iff $P(d) \in k$. Then, $v$ is closed upwards, since $k \leq k'$ implies $k \subseteq k'$. The strictness condition is satisfied, since each $k$ is closed under derivation; and the finite verification condition holds, since only finitely many predicates occur in $\Gamma$. Therefore $W \in \mathcal{W}$.
\begin{lemma}[Truth lemma]\label{lemma: Truth lemma}
  (i) For any $i \in \{ * \} \cup \mathbb{N}^+$ and $N \in \ClGN[\mathcal{L}(r_i)]$, $N \in r_i$ iff $r_i \models N$. (ii) For any $k \in K \backslash \{ r_*, r_1, r_2, ...\}$ and $B \in \ClForm[\mathcal{L}(k)]$, $B \in k$ iff $k \models B$.
\end{lemma}
\begin{proof}
  (i, ii) Simul-induction on $N$ and $B$. (i)'s base case is trivial. (ii)'s base case holds by definition of $W$.
  
  (i) Induction step. Recall that $\vdash_{\NSF} N \lor \neg N$ (lemma \ref{lemma: Basic facts of NSF}, v). $r_*$'s case will be similar to $r_i$'s with $i \in \mathbb{N}^+$.
  \begin{enumerate}
    \item [($\to$)] ($\implies$) Use $\{ N_1 \to N_2, N_1 \} \vdash_{\NSF} N_2$. ($\impliedby$) Contraposition. If $N_1 \to N_2 \notin r_i$, then $N_1 \in r_i$ and $\neg N_2 \in r_i$. Therefore $N_2 \notin k$ for any $k \geq r_i$, and hence $r_i \not \models N_1 \to N_2$.
    \item [($\neg$)] ($\implies$) Suppose $\neg C \in r_i$ with $C \in \ClForm[\mathcal{L}(r_i)]$, $k \in K$ and $k \models C$. Then, by requirement (3), take a $k'$ such that $k \leq k'$ and $r_i \preccurlyeq k'$. Then $C \in k'$ by (ii)'s hypothesis. This contradicts $\neg C \in k'_{\GN}$ by requirement (2). ($\impliedby$) Contraposition. Suppose $\neg C \notin r_i$. Then $r_i \cup \{ C \} \not \vdash_{\NSF} \bot$, and hence $\langle C, \bot \rangle$ is $\langle \sigma_n, \tau_n \rangle$ for some $n$. Therefore by construction (clause ii, b, \textit{i}), there is a direct successor $k'$ of $r_i$ according to $\leq$ such that $C \in k'$, implying $r_i \not \models \neg C$ by (ii)'s hypothesis. (In $r_*$'s case, since $\neg \neg C \in r_*$, $r_1 \cup \{ C \} \not \vdash_{\NSF} \bot$.)
    \item [($\forall$-glo)] ($\implies$) Suppose $d \in D$, $k \geq r_i$, and let $l$ be such that $d \in \ClTerm[\mathcal{L}(l)]$. Then by requirement (3), there is a $k' \geq k$ such that $l \preccurlyeq k'$. Use $\forall x N_1 \vdash_{\NSF} N_1[c/x]$ for all $c \in \ClTerm[\mathcal{L}(D)]$. ($\impliedby$) Contraposition. If $\forall x N_1 \notin r_i$, then $\neg \forall x N_1 \in r_i$. Here, $\neg \exists x \neg N \vdash_{\NSF} \forall x N$ in general, implying $\neg \forall x N \vdash_{\NSF} \exists x \neg N$. So there is a constant $a \in \mathcal{L}(r_i)$ such that $\neg N_1[a/x] \in r_i$. Therefore $k' \not \models N_1 [a/x]$ for any $k' \geq r_i$.
    \item [($\exists$-glo)] $r_i$ is a prime GN theory.
  \end{enumerate}
  
  (ii) Induction step. 
  \begin{enumerate}
    \item [($\to$)] ($\implies$) Suppose $C_1 \to C_2 \in k$, $k \leq k'$ and $k' \models C_1$. Then $\wneg \wneg C_2 \in k'$, and hence $k' \cup \{ C_2 \} \not \vdash_{\NSF} \bot$. So $\langle C_2, \bot \rangle$ is $\langle \sigma_n, \tau_n \rangle$ for some $n$. Therefore by construction (clause ii, b, \textit{i}), there is a direct successor $k''$ of $k'$ such that $C_2 \in k''$. ($\impliedby$) Contraposition. If $C_1 \to C_2 \notin k$, then $k \cup \{ C_1, \wneg C_2 \} \not \vdash_{\NSF} \bot$. So $\langle C_1 \land \wneg C_2, \bot \rangle$ is $\langle \sigma_n, \tau_n \rangle$ for some $n$. Look at the direct successor $k'$ of $k$ such that $C_1 \land \wneg C_2 \in k'$. Then $k' \models C_1$ and $k'' \not \models C_2$ for any $k'' \geq k'$.
    \item [($\neg$)] ($\implies$) Similar to (i). ($\impliedby$) $\neg C \notin k$ implies $\neg \neg C \in k$. Look at an $r_i \succ k$ by requirement (1).
    \item [($\forall$-loc)] ($\implies$) Similar to ($\to$). 
    ($\impliedby$) Contraposition. Suppose $\forall x C \notin k$. Then $\forall x C$ is $(\forall x B)_q$ for some $q$. Therefore by construction (clause ii, b, \textit{ii}), there is a successor $k'$ of $k$ such that $a \in \mathcal{L}(k')$ and $E(a) \land \wneg C(a) \in k'$. Therefore $k' \models E(a)$ and $k'' \not \models C(a)$ for any $k'' \geq k'$.
  \end{enumerate}
\end{proof}

\begin{proposition}[Strong completeness ($\NSF$)]\label{proposition: Strong completeness NSF}
  For any $\Gamma \subseteq_{\fin} \ClForm[\mathcal{L}]$ and $A \in \ClForm[\mathcal{L}]$, $\Gamma \models_{\mathcal{W}}^V A$ implies $\Gamma \vdash_{\NSF} A$.
\end{proposition}
\begin{proof}
  Contraposition. If $\Gamma \not \vdash_{\NSF} A$, then by the preceding lemmas, we can assure there is a $W \in \mathcal{W}$ with $r_0 \models_W \bigwedge \Gamma$, and $r_0 \not \models_W A$.
\end{proof}

\section{Study of prevalence}\label{section: Theory of prevalence}
In this section, we investigate the class of the `prevalent models', those with the `prevalence property': cf. section \ref{section: The conceptual constraints}, and see section \ref{section: Prevalent models} for the formal definition of prevalence. We call `$\SFQp$' the logic exhibited by this class, and show that it has natural extensions of the properties of Yamada \cite{Yamada2023}'s $\SF$. That paper showed that the assertibility in a prevalent model is calculable just as the classical truth-values (cf. \cite[proposition 2.5]{Yamada2023}), and that a formalisation of the bridging principle (\ref{principle: Conceptual identity}) between intuitionism and strict finitism holds with respect to the prevalent models (\cite[proposition 2.32]{Yamada2023}).


Section \ref{section: Prevalent models} collects basic properties of the prevalent models, and in section \ref{section: A complete natural deduction system NSFp} a natural deduction system $\NSFp$, sound and complete with respect to them, is provided. We show in section \ref{section: On interchangeability of negation} that $\neg$ and $\wneg$ are almost interchangeable in $\SFQp$, but not always. $\wneg$ is not needed to capture the theorems of $\NSFp$, and rather $\neg$ is essential to some of them (propositions \ref{proposition: Reducibility to neg} and \ref{proposition: Irreducibility to wneg}). This suggests that the $\neg$-less theorems of $\NSFp$ form an `intermediate' part in $\SFQp$ which can be captured only with intuitionistic connectives with a time-gap.

In section \ref{section: Relation between the neg-less fragment and intermediate logics}, we compare this part with intermediate logics. An intermediate logic is obtained by adding classical theorems to $\IQC$ and taking the closure under Modus Ponens and substitution for formulas. It is called `proper' if it is not classical first-order logic $\CQC$\footnote{
$\SFQp$ (or its $\neg$-less part) is not an intermediate logic, since it lacks Modus Ponens.
}.
In this context, we regard a logic as the set of its theorems. In section \ref{section: Comparison with intermediate logics}, $\CQC$ is shown to correspond to the $\neg$-less part under a simple double-negation translation (proposition \ref{proposition: neg neg A in NSFp is classical}); and we prove the proper intermediate logic we call `$\HTCD$' is embeddable to the part under another double-negation translation (proposition \ref{proposition: HTQ and NSFp}).

In section \ref{section: A prevalent model as an intuitionistic node}, we extend Yamada's ideas (\cite[cf. sections 1.4 and 2.5]{Yamada2023}) and provide a quantificational version of the formalised bridging principle (\ref{principle: Conceptual identity}). Similarly to the case of $\SF$, a set $\mathcal{U}$ of finite prevalent models can induce an intuitionistic model, if arranged by a partial order $\preceq$ of practical verifiability. We call $\langle \mathcal{U}, \preceq \rangle$ a `generation structure': the nodes $W \in \mathcal{U}$ typically represent different generations of the same agent with increasing verificatory power. We modify the semantics of $\SFQ$ to take into account generations, and define a new `generation semantics' to evaluate formulas in a generation structure. In the propositional case, the totality of the valid formulas in all generation structures coincides with intuitionistic propositional logic $\IPC$ (\cite[proposition 2.32]{Yamada2023}). We provide an estimation of the validity in general on the quantificational level in terms of $\IQC$ (proposition \ref{proposition: T Th G and IQC}), and present a class of generation structures whose validity coincides with $\IQC$ (proposition \ref{corollary: Th GC is IQC}). In section \ref{section: Ending remarks: Further topic}, we provide some remarks mainly on these results about the bridging principle, to end this article.

\subsection{Prevalent models}\label{section: Prevalent models}
In this section, we semantically define prevalence and its relevant notions. We see the models with the prevalence lose several distinctions, and the assertibility and the validity in them obey simple characterisations.

Fix a $W = \langle K, \leq, D, J, v \rangle \in \mathcal{W}$ with its root $r$. Prevalence is defined via two `prevalence properties'.
\begin{definition}[Prevalence]\label{definition: Prevalence}
  An $A \in \ClForm[\mathcal{L}(D)]$ is \textit{prevalent in $W$} (notation $\models_W^P A$) if for all $k \in K$, there is a $k' \geq k$ with $k' \models A$. $W$ has the \textit{formula prevalence property} if $\models_W^A A$ implies $\models_W^P A$ for all $A$; and has the \textit{object prevalence property} if $\models_W^P E(\overline{d})$ for all $d \in D$. A \textit{prevalent model} is one with both properties. We denote their class by $\mathcal{W}_\textup{P}$.
\end{definition}
\noindent We note that $\models_W^P A$ iff $\models_W^V \neg \wneg A$; $\not \models_W^P A$ iff $\models_W^V \neg \neg \wneg A$. The formula prevalence is equivalent to $\models_W^V \neg \neg A \to A$; and the object prevalence $\models_W^V \forall x \neg \wneg E(x)$. Also, if $\leq$ is total, then $W$ has the formula prevalence.

Conceptually, prevalence is a stronger notion of actual verifiability than assertibility. While an assertible statement is verified at at least one point in all possible histories, a prevalent statement is verified in the future of any point in each history. In this sense, all histories in a prevalent model are homogeneous: no matter how the agent proceeds in verifying, `what is open to verification' does not change.

One could formally replace the prevalence properties with other properties. The formula prevalence reduces to the \textit{atomic prevalence property} that $\models_W^A P$ implies $\models_W^P P$ for all $P \in \ClAtom[\mathcal{L}(D)]$ (including $E(c)$), which can be characterised in two ways.
\begin{proposition}
  (i) The atomic prevalence implies the formula prevalence. (ii) (a) $W$ has the atomic prevalence iff (b) $\models_W^V \neg \neg P \to P$ for all $P$ iff (c) $\models_W^V (P \to Q) \to ((P \to \neg Q) \to \neg P)$ for all $P$ and $Q$.
\end{proposition}
\begin{proof}
  (i) Induction. (ii) (a $\implies$ b) Easy. (b $\implies$ c) Use $\models_W^V \neg P \lor \neg \neg P$. (c $\implies$ a) Suppose (c). Contraposition. If $\not \models_W^P P$, then $k \models P \to \neg P$ for some $k$. Since $k \models P \to P$, (c) implies $\models_W^V \neg P$ and hence $\not \models_W^A P$.
\end{proof}
\noindent Further, either one of the two prevalence properties can be weakened (but not both at the same time). We say $W$ has the \textit{atomic decidability property} if $\models_W^V \forall x (P(x) \lor \neg P(x))$ for all $P \in \Atom[\mathcal{L}(D)]$; and the \textit{total constructibility property} if $\models_W^V \forall x \neg \neg E(x)$, in the sense that every object in the domain of discourse is constructed at some step. Plainly, the atomic prevalence implies the former; and the object prevalence the latter.
\begin{proposition}\label{proposition: Weakening the prev conds}
  (i) The atomic prevalence and the total constructibility imply the object prevalence. (ii) The atomic decidability and the object prevalence imply the atomic prevalence.
\end{proposition}
\begin{proof}
  Easy.
\end{proof}
\noindent We choose the combination $\neg \neg A \to A$ ($A \in \Fm$) and $\forall x \neg \neg E(x)$ as the axiomatisation for our proof theory (cf. section \ref{section: A complete natural deduction system NSFp}).

The prevalent models lose several semantic distinctions. If $W$ has the formula prevalence, then the assertibility is equivalent to the prevalence ($\models_W^A A$ iff $\models_W^P A$); $\models_W^A A$ iff $k \models_W A$ for any leaf $k$; and outermost $\neg$ and $\wneg$ are interchangeable ($k \models_W \wneg A$ iff $k \models_W \neg A$). We further investigate the interchangeability of $\neg$ and $\wneg$ in section \ref{section: On interchangeability of negation}.

Assume $W$ further has the object prevalence, and hence $W \in \mathcal{W}_\textup{P}$. Then the two modes of universal quantification merge: i.e. unconditionally, $k \models_W \forall x A$ iff $k \models_W \top \to A[\overline{d}/x]$ for all $d \in D$ iff $\models^A A[\overline{d}/x]$ for all $d \in D$. This yields that assertibility, as well as prevalence, is calculable the same way as the classical truth-values; and the validity conforms to the following scheme.
\begin{proposition}\label{proposition: Assertibility and validity in the prevalent models}
  For any $A, B \in \ClForm[\mathcal{L}(D)]$, the following hold for $W$.
  \begin{enumerate}
    \item (a) $\models^A A \land B$ iff $\models^A A$ and $\models^A B$. (b) $\models^A A \lor B$ iff $\models^A A$ or $\models^A B$. (c) $\models^A A \to B$ iff $\not \models^A A$ or $\models^A B$. (d) $\models^A \neg A$ iff $\not \models^A A$. (e) $\models^A \forall x A$ iff $\models^A A[\overline{d}/x]$ for all $d \in D$. (f) $\models^A \exists x A$ iff $\models^A A[\overline{d}/x]$ for some $d$.
    \item (a) $\models^V A \land B$ iff $\models^V A$ and $\models^V B$. (b) $\models^V A \lor B$ iff $\models^V A$ or $\models^V B$. (c) $\models^V A \to B$ iff $\models^A A \to B$. (d) $\models^V \neg A$ iff $\models^A \neg A$. (e) $\models^V \forall x A$ iff $\models^A \forall x A$. (f) $\models^V \exists x A$ iff $\models^V A[\overline{d}/x]$ for some $d$ for global quantification, and $\models^V E(\overline{d}) \land A[\overline{d}/x]$ for some $d$ for local quantification.
  \end{enumerate}
\end{proposition}
\begin{proof}
  Use the two prevalence properties. For (ii, a-b \& f), look at the root.
\end{proof}
\noindent We note that then plainly $\models_{\mathcal{W}_\textup{P}}^V \forall x (A \lor \neg A)$.

Proposition \ref{proposition: Assertibility and validity in the prevalent models} (i) means that what holds in the future obeys classical logic. In other words, the collection of the assertible formulas can comprise a classical node. This may lead one to the idea that a prevalent model can be summarised by two nodes, one leaf node for the assertible formulas and the root node for the valid formulas. This indeed is possible: formally, the semantic consequence in $\mathcal{W}_\textup{P}$ boils down to that in $\mathcal{W}_{\textup{P}2}$, the subclass of the two-node prevalent models\footnote{
This result corresponds to \cite{Yamada2023}'s proposition 2.14.
}.
To prove this (proposition \ref{proposition: Strong contraction}), we will use that the prevalent models are closed under `model contraction' and have the submodel property, as follows.

Given an $A \in \ClForm[\mathcal{L}(D)]$, we define a \textit{contraction model $W \vert A$ with respect to $A$} to be $\langle \{ r, k \}, \leq', D, J, v' \rangle$ as follows. $r <' k$. For all $P \in \Pred$, if $P$ occurs in $A$ or is $E$, then let $P^{v'(r)} = P^{v(r)}$ and $P^{v'(k)} = \bigcup_{l \in K} P^{v(l)}$; otherwise, $P^{v'(r)} = P^{v'(k)} = \emptyset$. Plainly, $W \vert A$ satisfies the finite verification condition (right above definition \ref{definition: Actual verification relation models_W}); and $W \vert A \in \mathcal{W}_{\textup{P}2}$.

\begin{lemma}[Model contraction]\label{lemma: Model contraction}
  For any subformula $A'$ of $A$, (i) $\models_W^A A'$ iff $k \models_{W \vert A} A'$, and (ii) $\models_W^V A'$ iff $r \models_{W \vert A} A'$.
\end{lemma}
\begin{proof}
  (i, ii) Induction. Use proposition \ref{proposition: Assertibility and validity in the prevalent models}. For (ii), use (i) also.
\end{proof}
\noindent For each $k \in K$, we define the \textit{submodel $W_k \in \mathcal{W}$ generated by $k$} to be ${\langle K_k, \leq_k, D, J, v_k \rangle}$, where $K_k = \{ k' \in K \mid k \leq k' \}$, $\leq_k$ is the restriction of $\leq$ to $K_k$ and $P^{v_k(k')} = P^{v(k')}$. Confirm that $W_k \in \mathcal{W}_\textup{P}$, given $W \in \mathcal{W}_\textup{P}$; and by induction that for all $k' \in K_k$ and $A \in \ClForm[\mathcal{L}(D)]$, $k' \models_{W_k} A$ iff $k' \models_W A$.
\begin{proposition}\label{proposition: Strong contraction}
  For any $\Gamma \subseteq_{\fin} \ClForm[\mathcal{L}]$, $\Gamma \models_{\mathcal{W}_\textup{P}}^V A$ iff $\Gamma \models_{\mathcal{W}_\textup{P2}}^V A$.
\end{proposition}
\begin{proof}
  ($\impliedby$) Use lemma \ref{lemma: Model contraction} (ii). If $\Gamma = \emptyset$, look at the contraction $W \vert A$ for a given $W \in \mathcal{W}_{\textup{P}}$. If $\Gamma \neq \emptyset$, then suppose $k \in K$ and $k \models_W \bigwedge \Gamma$, and consider the contraction $W_k \vert (\bigwedge \Gamma \land A)$. 
\end{proof}
\noindent This result suggests that the order of construction and verification is irrelevant to the semantic consequence of strict finitistic reasoning with prevalence. The models in $\mathcal{W}_\textup{P2}$ are obtained by discarding the information of such orders.

Further, the next result shows that, for strict finitistic validity, we can assume that any verification history starts with zero objects constructed. Formally, the validity in $\mathcal{W}_{\textup{P}2}$ boils down to that in $\mathcal{W}_{\textup{P}2 \emptyset} \subseteq \mathcal{W}_{\textup{P}2}$, the subclass of the `preconstructive' models. We say that a $W = \langle K, \leq, D, J, v \rangle \in \mathcal{W}$ with root $r$ is \textit{preconstructive} if $r \not \models_W E(c)$ for any $c \in \ClTerm[\mathcal{L}(D)]$. For each $W \, (\in \mathcal{W}_{\textup{P}2}) = \langle \{ r, k \}, \leq, D, J, v \rangle$, we let $W_{\emptyset} = \langle \{ r, k \}, \leq, D, J, v_\emptyset \rangle$, where $P^{v_\emptyset (r)} = \emptyset$ and $P^{v_\emptyset (k)} = P^{v (k)}$ for all $P \in \Pred$. Then $W_{\emptyset} \in \mathcal{W}_{\textup{P}2 \emptyset}$.
\begin{proposition}\label{proposition: Reduction of Wp to Wp2ets}
  (i) $k \models_{W_\emptyset} A$ iff $k \models_W A$. (ii) $r \models_{W_\emptyset} A$ implies $r \models_W A$. (iii) $\models_{\mathcal{W}_{\textup{P}2}}^V A$ iff $\models_{\mathcal{W}_{\textup{P}2 \emptyset}}^V A$.
\end{proposition}
\begin{proof}
  (i, ii) Induction. (ii) uses (i). (iii, $\impliedby$) The contraposition holds by (ii).
\end{proof}

In $\mathcal{W}_\textup{P}$, more formulas are stable than in $\mathcal{W}$. We define the class $\ST_\textup{P}[\mathcal{L}]$ by $\ST_\textup{P}[\mathcal{L}] ::= S \mid (S' \lor S') \mid (\exists x S')$, where $S \in \ST[\mathcal{L}]$, and $\exists x S'$ is global. Then, $\ST[\mathcal{L}] \subseteq \ST_\textup{P}[\mathcal{L}]$.

\begin{lemma}
  Any $\ST_\textup{P}$ formula is stable in all $W \in \mathcal{W}_{\textup{P}}$.
\end{lemma}
\begin{proof}
  Induction.
\end{proof}
\noindent We provide a characterisation of the validity in $\mathcal{W}_\textup{P}$ in terms of $\ST_\textup{P}$ and $\CQC$ in proposition \ref{proposition: neg neg A in NSFp is classical} and corollary \ref{proposition: When a classical theorem is a theorem}.

\subsection{A complete natural deduction system \textbf{NSF$_\textup{P}$}}\label{section: A complete natural deduction system NSFp}
We define a natural deduction system $\NSFp$ by all the rules of \textbf{NSF} with the following modifications.
\begin{multicols}{2}
  \begin{prooftree}
    \AxiomC{ \, }
    \RightLabel{DNE}
    \UnaryInfC{$\neg \neg A \to A$}
  \end{prooftree}
  \begin{prooftree}
    \AxiomC{ \, }
    \RightLabel{Obj}
    \UnaryInfC{$\forall x \neg \neg E(x)$}
  \end{prooftree}
\end{multicols}
\begin{multicols}{2}
  \begin{prooftree}
    \AxiomC{$\wneg \wneg S$}
    \RightLabel{ST$_\textup{P}$}
    \UnaryInfC{$S$}
  \end{prooftree}
  \begin{prooftree}
    \AxiomC{$\wneg A$}
    \RightLabel{$\neg$ I$_\textup{P}$}
    \UnaryInfC{$\neg A$}
  \end{prooftree}
\end{multicols}
\noindent (i) We have prevalence axioms (DNE) and (Obj) (cf. proposition \ref{proposition: Weakening the prev conds}, i). (ii) (ST) is extended to (ST$_\textup{P}$) that is applicable to $\wneg \wneg S$ with any $S \in \ST_\textup{P}$. (iii) Instead of ($\neg$ I$_1$), ($\neg$ I$_2$) and ($\to \bot$ E) of $\NSF$, we have ($\neg$ I$_\textup{P}$). (iv) We drop ($\forall$-loc I) and ($\forall$-loc E); and call ($\forall$-glo I) `($\forall$ I$_\textup{P}$)', and ($\forall$-glo E) `($\forall$ E$_\textup{P}$)', and use them also for formulas with a term not occurring only globally. (Their variable conditions stay.) We write $\vdash_{\NSFp}$ for derivability in $\NSFp$.

Plainly, $\NSFp$ is sound with respect to $\mathcal{W}_\textup{P}$. We note that rules ($\neg$ I$_1$), ($\neg$ I$_2$) and ($\to \bot$ E) of $\NSF$ are derivable via ($\neg$ I$_\textup{P}$); also ($\forall$-loc I) via ($\forall$ I$_\textup{P}$), and ($\forall$-loc E) via ($\forall$ E$_\textup{P}$). Therefore $\Gamma \vdash_{\NSF} A$ implies $\Gamma \vdash_{\NSFp} A$. Also, in the presence of (DNE) and (Obj), $\vdash_{\NSFp} \wneg \wneg E(t)$ for all $t \in \Term$. So the following rules are derivable for formulas with a term not occurring only globally:
\begin{multicols}{2}
  \begin{prooftree}
    \AxiomC{$ \, $}
    \noLine
    \UnaryInfC{$ \, $}
    \noLine
    \UnaryInfC{$ \, $}
    \noLine
    \UnaryInfC{$ \, $}
    \noLine
    \UnaryInfC{$ \, $}
    \noLine
    \UnaryInfC{$ \, $}
    \noLine
    \UnaryInfC{$ \, $}
    \noLine
    \UnaryInfC{$ A[t/x] $}
    \RightLabel{$\exists$-loc I$_\textup{P}$}
    \UnaryInfC{$\wneg \wneg \exists x A$}
  \end{prooftree}
  \begin{prooftree}
    \AxiomC{$\exists y A$}
    \AxiomC{$[A[x/y]]$}
    \noLine
    \UnaryInfC{$\vdots$}
    \noLine
    \UnaryInfC{$\wneg \wneg C$}
    \RightLabel{$\exists$-loc E$_\textup{P}$}
    \BinaryInfC{$\wneg \wneg C,$}
  \end{prooftree}
\end{multicols}
\noindent with the same variable conditions as ($\exists$-loc I) and ($\exists$-loc E), respectively.

We will use some of the following derivabilities in the completeness proof (proposition \ref{proposition: Strong completeness NSFp}).
\begin{lemma}
  (i) $\neg \neg A \vdash_{\NSFp} \neg \wneg A$. (ii) $\neg \neg A \land \neg \neg B \vdash_{\NSFp} \neg \neg (A \land B)$. (iii) $\neg \neg A \land \neg B \vdash_{\NSFp} \neg (A \to B)$. (iv) $\neg \neg (A \to B) \vdash_{\NSFp} \neg A \lor \neg \neg B$. (v) $\vdash_{\NSFp} S \lor \wneg S$, where $S \in \ST_\textup{P}$. (vi) $\neg \exists x \neg A \vdash_{\NSFp} \forall x A$. (vii) $\neg \neg \forall x A \vdash_{\NSFp} \forall x A$. (viii) $\exists x \neg \neg A \vdash_{\NSFp} \neg \neg \exists x A$.
\end{lemma}
\begin{proof}
  (i) By (DNE). (ii) Use (i) and that $\neg \wneg A \land \neg \wneg B \vdash_{\NSF} \neg \wneg (A \land B)$. (iii) By ($\neg$ I$_\textup{P}$). (iv) Use the contraposition of (iii). (v) Follows from $\vdash_{\NSF} \wneg \wneg (A \lor \wneg A)$ (lemma \ref{lemma: Basic facts of NSF}, iii). (vi) Confirm $\neg \exists x \neg A \vdash_{\NSFp} \neg \neg A(y)$. Then apply (DNE). (vii) By (DNE). (viii) First, we have $\neg \neg A(y) \vdash_{\NSFp} \neg \neg \neg \neg \exists x A$: since $A(y) \vdash_{\NSFp} \neg \neg \exists x A$ by ($\exists$-loc I$_\textup{P}$), it suffices to apply contraposition twice. So apply ($\exists$-glo E) to $\exists x \neg \neg A$. 
  
\end{proof}

The proof of the completeness of $\NSFp$ is in the Henkin-style. It is simpler than that of $\NSF$.
\begin{proposition}[Strong completeness ($\NSFp$)]\label{proposition: Strong completeness NSFp}
  For any $\Gamma \subseteq_{\fin} \ClForm[\mathcal{L}]$ and $A \in \ClForm[\mathcal{L}]$, $\Gamma \models_{\mathcal{W}_\textup{P}}^V A$ implies $\Gamma \vdash_{\NSFp} A$.
\end{proposition}
\noindent \textit{Proof.} Contraposition. Given $\Gamma \not \vdash_{\NSFp} A$, extend $\Gamma$ to a $\Gamma' \subseteq \ClForm[\mathcal{L}'_{\vert \Gamma}]$ in an extended language $\mathcal{L}' = \mathcal{L}(M)$ in the same way as lemma \ref{lemma: Extension lemma}. Confirm that, in the present context too, $\Gamma'$ is a prime theory with respect to $\mathcal{L}'_{\vert \Gamma}$ such that $\Gamma' \not \vdash_{\NSFp} A$.

We define $W = \langle \{r, k \}, \leq, D, J, v \rangle$ as follows. Let $r < k$. Define $D := \ClTerm[\mathcal{L}'_{\vert \Gamma}]$. Inductively, for each constant $a$ in $\mathcal{L}'_{\vert \Gamma}$, $J^1(a) := a$; for each $d \in D$, $J^1(\overline{d}) := d$; for each $f \in \Func$, $(J^2 (f)) (d) := f(d)$ for all $d \in D$; and for each $f(c) \in \ClTerm[\mathcal{L}(D)]$, $J^1( f(c) ) := (J^2(f)) (J^1(c))$. For all $P \in \Pred$ and $d \in D$, $\langle d \rangle \in P^{v(r)}$ iff $P(d) \in \Gamma'$, and $\langle d \rangle \in P^{v(k)}$ iff $\neg \neg P(d) \in \Gamma'$. Then $W$ satisfies the finite verification condition, since $\Gamma$ is finite. Further, $\neg \neg E(d) \in \Gamma'$ for all $d \in D$ by (Obj). Therefore $W \in \mathcal{W}_{\textup{P}2}$.

We will show by induction that for all $B \in \ClForm[\mathcal{L}'_{\vert \Gamma}]$, (i) $k \models_W B$ iff $\neg \neg B \in \Gamma'$, and (ii) $r \models_W B$ iff $B \in \Gamma'$. Recall that $\vdash_{\NSF} N \lor \neg N$.

(i) The base case is trivial. ($\land$) Use $\neg \neg (C \land D) \dashv \vdash_{\NSFp} \neg \neg C \land \neg \neg D$. ($\lor$) Use $\neg \neg (C \lor D) \dashv \vdash_{\NSF} \neg \neg C \lor \neg \neg D$, and that $\Gamma'$ is a prime theory. ($\to$) Use $\neg \neg (C \to D) \dashv \vdash_{\NSFp} \neg C \lor \neg \neg D$. ($\neg$) Use $\neg \neg \neg C \dashv \vdash_{\NSF} \neg C$.
\begin{itemize}
  \item [($\forall$)] By the hypothesis, $k \models_W \forall x C$ iff $k \models_W C[\overline{d}/x]$ for all $d \in D$ iff $\neg \neg C[\overline{d}/x] \in \Gamma'$ for all $d$. So it suffices to show that the last is equivalent to $\neg \neg \forall x C \in \Gamma'$. ($\implies$) $\neg \exists x \neg C \in \Gamma'$ holds. For, if $\exists x \neg C \in \Gamma'$, then there is an $a \in M$ such that $\neg C[a/x] \in \Gamma'$, but this is contradictory. Now use $\neg \exists x \neg C \vdash_{\NSFp} \forall x C$. ($\impliedby$) Use $\neg \neg \forall x C \vdash_{\NSFp} \forall x C$.
  \item [($\exists$)] Since $k$ is a leaf, we can ignore the modes of occurrence of $x$. So we have by the hypothesis, $k \models_W \exists x C$ iff $k \models_W C[\overline{d}/x]$ for some $d \in D$ iff $\neg \neg C[d/x] \in \Gamma'$ for some $d$. It suffices to show that the last is equivalent to $\neg \neg \exists x C \in \Gamma'$. ($\implies$) $\neg \forall x \neg C \in \Gamma'$ holds, since if $\forall x \neg C \in \Gamma'$, then $\neg C[d/x] \in \Gamma'$ for all $d$, a contradiction. Use $\neg \forall x \neg C \vdash_{\NSFp} \neg \neg \exists x C$. ($\impliedby$) Use that $\neg \neg \exists x C \vdash_{\NSF} \exists x \neg \neg C$.
\end{itemize}

(ii) The base case is trivial. ($\land$), ($\lor$) and ($\neg$) are easy. For ($\to$) and ($\forall$), use proposition \ref{proposition: Assertibility and validity in the prevalent models}, (i) above and $\neg \neg A \vdash_{\NSFp} \wneg \wneg A$.

Therefore $\Gamma \not \models_{\mathcal{W}_\textup{P}}^V A$, since $r \models_W \bigwedge \Gamma$, and $r \not \models_W A$.

\qed

\noindent This implies that $\vdash_{\NSFp} A$ iff $\models_{\mathcal{W}_{\textup{P}2 \emptyset}}^V A$ by propositions \ref{proposition: Strong contraction} and \ref{proposition: Reduction of Wp to Wp2ets}.

\subsection{On interchangeability of negation}\label{section: On interchangeability of negation}
In this section, we investigate the condition under which $\neg$ and $\wneg$ are interchangeable in the context of $\mathcal{W}_{\textup{P}}$. Outermost $\neg$ and $\wneg$ are interchangeable, as we noted after proposition \ref{proposition: Weakening the prev conds}. In fact, as we will show, they are interchangeable almost everywhere (proposition \ref{proposition: Replace neg with wneg}); and $\wneg$ is in a sense reducible to $\neg$ on the level of the theorems (proposition \ref{proposition: Reducibility to neg}), but not vice versa (proposition \ref{proposition: Irreducibility to wneg}). As a result, we see that there is a set of the theorems of $\SFQp$ which we do not need strict finitistic negation to `express'. This result motivates us to focus on this set, and compare it with intermediate logics, as we conduct in the next section (\ref{section: Relation between the neg-less fragment and intermediate logics}).

We first establish the conditions for interchangeability on the level of the models, and then proceed to the level of the theorems. However, to treat substitution of $\neg$ and $\wneg$ properly, we need the notion of a subformula occurrence (cf. \cite[pp.59-61]{TroelstravanDalen1988}). We start with its definition and the basic facts.

We define the set $\Cont[\mathcal{L}_0]$ of the expressions $F$, called the \textit{formula contexts}, on any given language $\mathcal{L}_0$ as follows.
\begin{itemize}
  \item $\Cont[\mathcal{L}_0] ::= * \mid (F \square A) \mid (A \square F) \mid (\neg F) \mid (Qx F)$,
\end{itemize}
where $*$ is a primitive symbol, $A \in \Fm[\mathcal{L}_0]$, $\square \in \{ \land, \lor, \to \}$ and $Q \in \{ \forall, \exists \}$. We may omit the parentheses. The order of precedence is the same as the formulas. $\ClCont[\mathcal{L}_0]$ is the set of the closed formula contexts.

Intuitively, a formula context $F$ is a formula with one placeholder symbol $*$. If a formula $B$ fills the $*$, it turns to a formula in the usual sense, which we denote by $F[B]$; and if another $C$ does, $F[C]$ is obtained. We use this machinery to treat the substitution of $C$ for $B$ occurring in a formula $F[B]$.
\begin{definition}[{$[B/*]$}, formula occurrence]
  Substitution $[B/*]$ (or simply $[B]$) of a $B \in \Fm[\mathcal{L}_0]$ for $*$ is defined inductively on $\Cont[\mathcal{L}_0]$ by $*[B] := B$, $(F \square A)[B] := F[B] \square A$, $(A \square F)[B] := A \square F[B]$, $(\neg F)[B] := \neg (F[B])$ and $(QxF)[B] := Qx (F[B])$.
\end{definition}
\noindent We call any pair $\langle F, B \rangle \in \Cont[\mathcal{L}_0] \times \Fm[\mathcal{L}_0]$ a \textit{formula occurrence}, as it specifies the occurrence, in this case, of $B$ in the formula $F[B]$. We understand that $F[C]$ serves as the result of replacing $B$ in $F[B]$ with a $C \in \Fm[\mathcal{L}_0]$. So we define the result $A[C/\langle F, B \rangle]$ of substitution of an occurrence of $B$ in $A$ with $C$ to be $F[C]$ if $F[B] = A$, and $A$ otherwise.

For brevity, however, we will directly be discussing $F[C]$ etc. to speak of $A[C/\langle F, B \rangle]$ etc. We note that a variable in $C$ may become bound in $F[C]$ (cf. \textit{ibid.}, section 3.10). Plainly, if $x$ is occurring in $F[\wneg B]$ only globally, so is in $F[\neg B]$. 

Also, we define substitution $[s/t]$ for $s, t \in \Term[\mathcal{L}_0]$ in an $F \in \Cont[\mathcal{L}_0]$ the same way as formulas, except $*[s/t] := *$. We note that then $(F[B])[s/t] = (F[s/t])[B[s/t]]$ does not in general hold. 
But $(F[B])[c/x] = (F[c/x])[B[c/x]]$ holds for any $c \in \ClTerm[\mathcal{L}_0]$, if $F[B]$ has no quantificational subformula $QxA$. 

Let us fix a $W = \langle K, \leq, D, J, v \rangle \in \mathcal{W}$. The following generalises the fact in the context of $\mathcal{W}$ that $k \models_W QxB$ iff $k \models_W Qy(B[y/x])$ (cf. after definition \ref{definition: Actual verification relation models_W}).
\begin{lemma}\label{lemma: Bound variables are always unique}
  For any $F[QxB] \in \ClForm[\mathcal{L}(D)]$ and $y$ not occurring in $F$ or $B$, $k \models_W F[QxB]$ iff $k \models_W F[Qy(B[y/x])]$.
\end{lemma}
\begin{proof}
  Prove by induction on $F$ that for any $F \in \Cont[\mathcal{L}(D)]$, $QxB \in \Fm[\mathcal{L}(D)]$ and sequence $\vec{c}$ in $\ClTerm[\mathcal{L}(D)]$, $k \models (F[QxB])[\vec{c}/\vec{z}]$ iff $k \models (F[Qy(B[y/x])])[\vec{c}/\vec{z}]$, where $\vec{z}$ is the sequence of $\FV(F[QxB])$. ($\forall$) Let $F = \forall u F'$. Consider $k \models (\forall u (F'[QxB])) [\vec{c}/\vec{z}]$. This formula is $\forall u ((F'[QxB]) [\vec{c}/\vec{z}])$, since no element of $\vec{z}$ is $u$. This is equivalent to that for all $d$, $k \models \top \to ((F'[QxB])[\vec{c}/\vec{z}] ) [\overline{d}/u]$ (or with $E(\overline{d})$ instead of $\top$). This formula is $\top \to (F'[QxB]) [\vec{c}/ \vec{z}, \overline{d}/u]$. Apply the hypothesis.
\end{proof}
\noindent So we will always be assuming the bound variables are unique.

Now, assume $W \in \mathcal{W}_{\textup{P}}$. Then we see $\neg$ and $\wneg$ can freely be exchanged preserving assertibility.
\begin{proposition}\label{proposition: wneg and neg are the same in assertibility}
  For any $F[B] \in \ClForm[\mathcal{L}(D)]$, $\models_W^A F[\neg B]$ iff $\models_W^A F[\wneg B]$.
\end{proposition}
\begin{proof}
  Induction. Use proposition \ref{proposition: Assertibility and validity in the prevalent models}. 
  ($\forall$) Let $F = \forall x F'$. Then $\models^A (\forall x F')[\neg B]$ iff $\models^A (F'[\neg B]) [\overline{d}/x]$ for all $d$. Since the bound variables are unique, no $QxA$ occurs in $F'[\neg B]$. So the last formula is $(F'[\overline{d}/x]) [\neg (B[\overline{d}/x])]$. By the hypothesis, then, the last is equivalent to that $\models^A (F'[\overline{d}/x]) [\wneg (B[\overline{d}/x])]$ for all $d$.
\end{proof}
\noindent We see that on the nodes' level, $\wneg$ can always be replaced by $\neg$.
\begin{corollary}\label{corollary: wneg to neg}
  $k \models_W F[\wneg B]$ implies $k \models_W F[\neg B]$.
\end{corollary}
\begin{proof}
  Similar to proposition \ref{proposition: wneg and neg are the same in assertibility}. Use it for ($\to$), ($\neg$) and ($\forall$). ($\exists$-glo) If $\exists x (F'[\wneg B])$ is global, then so is $\exists x (F'[\neg B])$. 
\end{proof}
\noindent We have a characterisation of the class of the formula occurrences for which the converse holds, although it may not be so illuminating, regrettably.
\begin{definition}[$\mathcal{M}_k$]\label{definition: Mk}
  For each $k \in K$, the set $\mathcal{M}_k \subseteq \Cont[\mathcal{L}(D)] \times \Fm[\mathcal{L}(D)]$ of formula occurrences is the smallest such that for all $F \in \Cont[\mathcal{L}(D)]$ and $A, B \in \Fm[\mathcal{L}(D)]$,
  \begin{itemize}
    \item $\langle *, B \rangle, \langle F \to A, B \rangle, \langle A \to F, B \rangle, \langle \neg F, B \rangle, \langle \forall x F, B \rangle \in \mathcal{M}_k$;
    \item $\langle F \land A, B \rangle, \langle A \land F, B \rangle \in \mathcal{M}_k$ if $\langle F, B \rangle \in \mathcal{M}_k$;
    \item $\langle F \lor A, B \rangle, \langle A \lor F, B \rangle \in \mathcal{M}_k$ if $\langle F, B \rangle \in \mathcal{M}_k$ or $k \models_W A$; and
    \item $\langle \exists x F, B \rangle \in \mathcal{M}_k$ if
    \begin{enumerate}
      \item (a) $\exists x (F[\neg B])$ is local or (b) $\exists x (F[\wneg B])$ is global, and
      \item (a) $k \models (F[\neg B])[\overline{d}/x]$ implies $\langle F[\overline{d}/x], B[\overline{d}/x] \rangle \in \mathcal{M}_k$ for all $d \in D$, or (b) $k \models \exists x (F[\wneg B])$.
    \end{enumerate}
  \end{itemize}
\end{definition}

\begin{proposition}\label{proposition: Replace neg with wneg}
  Let $F[B] \in \ClForm[\mathcal{L}(D)]$. (i) If $\langle F, B \rangle \in \mathcal{M}_k$ and $k \models_W F[\neg B]$, then $k \models_W F[\wneg B]$. (ii) (i)'s converse holds, if $k \not \models_W E(c)$ for any $c \in \ClTerm[\mathcal{L}(D)]$.
\end{proposition}
\begin{proof}
  Both induction on $F$. (i) Use proposition \ref{proposition: wneg and neg are the same in assertibility} for ($\to$), ($\neg$) and ($\forall$). (ii) By corollary \ref{corollary: wneg to neg}, it suffices to prove that $\langle F, B \rangle \in \mathcal{M}_k$ is implied by that $k \not \models_W E(c)$ for any $c \in \ClTerm[\mathcal{L}(D)]$. ($\exists$) One can assume the quantification is global.
\end{proof}

Thus, $\wneg$ can always be $\neg$ (lemma \ref{corollary: wneg to neg}); and $\neg$ can be $\wneg$ in $A$, if $A$ is analysable into an `$\mathcal{M}_k$-format', i.e., $A$ is $F[\neg B]$ for some $\langle F, B \rangle \in \mathcal{M}_k$ (proposition \ref{proposition: Replace neg with wneg}). It follows that $\neg$ and $\wneg$ are always interchangeable in the propositional case.
\begin{corollary}\label{corollary: neg and wneg are totally interchangeable in the qf-free case}
  If $F[B] \in \ClForm[\mathcal{L}(D)]$ is quantifier-free, then $k \models F[\neg B]$ iff $k \models F[\wneg B]$.
\end{corollary}
\begin{proof}
  Prove by induction that $k \models F[\neg B]$ implies $\langle F, B \rangle \in \mathcal{M}_k$.
\end{proof}
\noindent This result corroborates \cite{Yamada2023}'s identification of $\neg$ with $\to \bot$.

Now, we move on to the level of the theorems. Let us simply write $\SFQp$ for the set of the theorems of $\NSFp$; and $\SFQ_{\textup{\textbf{P}}}^{- \wneg}$ for the set of those not involving the combination $\to \bot$. We will show that $\wneg$ is not needed to capture $\SFQp$: any element in $\SFQp$ is obtained by changing suitable $\neg$ in some element in $\SFQ_{\textup{\textbf{P}}}^{-\wneg}$ into $\wneg$.

Define $\mathcal{M} := \bigcap_r \mathcal{M}_r$, where $r$ is a metavariable for the root nodes of $W \in \mathcal{W}_{\textup{P}2 \emptyset}$ (cf. definition before proposition \ref{proposition: Reduction of Wp to Wp2ets}).
For each $A \in \ClForm[\mathcal{L}]$, define $\Phi(A) := \{ F[\wneg B] \mid \langle F, B \rangle \in \mathcal{M}$ and $A = F[\neg B] \}$. Roughly speaking, $\Phi(A)$ is the set of `$\wneg$-variants' of $A$. To obtain $\Phi(A)$, one lists all the ways in which $A$ is analysable into an $\mathcal{M}$-format $F[\neg B]$, and in each of the ways, changes $\neg B$ occurring in $A$ into $\wneg B$. For instance, $\Phi(\neg A \lor \neg \neg A) = \{ \wneg A \lor \neg \neg A, \neg A \lor \wneg \neg A, \neg A \lor \neg \wneg A \}$. Then we see that $\SFQp$ coincides with the set $\SFQ_{\textup{\textbf{P}}}^{-\wneg}$ extended by arbitrary finite iterations of $\Phi$.
\begin{proposition}\label{proposition: Reducibility to neg}
  (i) If $\vdash_{\NSFp} F[\wneg B]$, then $\langle F, B \rangle \in \mathcal{M}$. (ii) $\SFQp = \bigcup_{n = 0}^\infty \Phi^{n} [\SFQ_{\textup{\textbf{P}}}^{-\wneg}]$.
\end{proposition}
\begin{proof}
  (i) By soundness, $\models_{\mathcal{W}_{\textup{P}2\emptyset}}^V F[\wneg B]$. Use proposition \ref{proposition: Replace neg with wneg} (ii).
  
  (ii) Prove by induction on $n$ that $A \in \SFQp$ has $n$ $\wneg$'s iff $A \in \Phi^{n} [\SFQ_{\textup{\textbf{P}}}^{-\wneg}]$. ($\implies$) Let $A \in \SFQp$ have $n+1$ $\wneg$'s. Then by (i), there is an $\langle F, B \rangle \in \mathcal{M}$ such that $A = F[\wneg B]$. By corollary \ref{corollary: wneg to neg} and completeness, $F[\neg B] \in \SFQp$. Since it has $n$ $\wneg$'s, $F[\neg B] \in \Phi^{n}[\SFQ_{\textup{\textbf{P}}}^{-\wneg}]$. Therefore $A = F[\wneg B] \in \Phi^{n+1}[\SFQ_{\textup{\textbf{P}}}^{-\wneg}]$.
  
  ($\impliedby$) Let $A \in \Phi^{n+1}[\SFQ_{\textup{\textbf{P}}}^{-\wneg}]$. Then there is an $\langle F, B \rangle \in \mathcal{M}$ such that $A = F[\wneg B]$, and by the hypothesis $F[\neg B] \in \SFQp$ and it has $n$ $\wneg$'s. Since $\langle F, B \rangle \in \mathcal{M}$, $\models_{\mathcal{W}_{\textup{P}2\emptyset}}^V F[\neg B]$ implies $\models_{\mathcal{W}_{\textup{P}2\emptyset}}^V F[\wneg B]$ by proposition \ref{proposition: Replace neg with wneg} (i). So $\vdash_{\NSFp} F[\wneg B]$ by completeness. 
\end{proof}

The same however does not hold for $\SFQ_{\textup{\textbf{P}}}^{-\neg}$, the set of the theorems without $\neg$. Let $\Psi(A) = \{ F[\neg B] \mid A = F[\wneg B] \}$.
\begin{proposition}\label{proposition: Irreducibility to wneg}
  $\bigcup_{n=0}^{\infty} \Psi^{n} [\SFQ_{\textup{\textbf{P}}}^{- \neg}] \subsetneq \SFQp$.
\end{proposition}
\begin{proof}
  ($\subseteq$) By induction, for any $n$, $A \in \Psi^{n} [\SFQ_{\textup{\textbf{P}}}^{-\neg}]$ implies that $A \in \SFQp$ and it has $n$ $\neg$'s. Use that $\vdash_{\NSFp} F[\wneg B]$ implies $\vdash_{\NSFp} F[\neg B]$. ($\neq$) Consider $\exists x (\neg P(x) \to \neg P(x)) \in \SFQp$ and $\exists x (\wneg P(x) \to \wneg P(x)) \notin \SFQp$. 
\end{proof}
\noindent So $\SFQp$ does not reduce to $\SFQ_{\textup{\textbf{P}}}^{- \neg}$ in the sense that the complement of $\bigcup_{n=0}^{\infty} \Psi^{(n)} [\SFQ_{\textup{\textbf{P}}}^{- \neg}]$ is not empty, although we have not succeeded in characterising the complement recursively. 

Thus intuitively, while the expansion of the set of the $\wneg$-less theorems reaches $\SFQp$, that of the $\neg$-less theorems ($\SFQ_{\textup{\textbf{P}}}^{-\neg}$) does not. This suggests that $\neg$ is essential to some theorems, specifically, to those in the complement of $\bigcup_{n=0}^{\infty} \Psi^{(n)} [\SFQ_{\textup{\textbf{P}}}^{- \neg}]$, such as $\exists x (\neg P(x) \to \neg P(x))$. One may further expect that $\SFQ_{\textup{\textbf{P}}}^{- \neg}$ is approximable by intermediate logics, since its only strict finitistic factors are implication and universal quantification, which are intuitionistic connectives with a time-gap. This is the next section's subject.

\subsection{Relation between the $\neg$-less fragment and intermediate logics}
\label{section: Relation between the neg-less fragment and intermediate logics}

In this section, we look into the relations between $\SFQ_{\textup{\textbf{P}}}^{- \neg}$ and intermediate logics. It turns out that they have close connections, and one might be tempted to call the former the `intermediate' part of $\SFQp$. We hope the results in what follows shed light on how $\SFQp$ exceeds the part that is related to intermediate logics.

We will be identifying a logic with the set of its theorems. For a logic $\textbf{X}$, we write $\vdash_{\textup{\textbf{X}}}$ for derivability in an intended proof system. That of $\SFQp$ is $\NSFp$. By $\textbf{X} + A$, we mean the logic $\textbf{X}$ with $A$ as an axiom: i.e., $B \in \textbf{X} + A$ iff $A \vdash_{\textbf{X}} B$.

Let language $\mathcal{L}^{-\neg}$ be $\mathcal{L}$'s $\neg$-less fragment, and $\mathcal{L}^{-\neg E}$ be $\mathcal{L}^{-\neg}$'s $E$-less fragment. Since intermediate logics do not involve $E$, we assume $\CQC \subseteq \ClForm[\mathcal{L}^{-\neg E}]$ etc., and mainly discuss the $E$-less theorems of $\CQC$ etc. The set of the theorems of $\SFQp$ without $\neg$ and $E$ is denoted by $\SFQ_{\textup{\textbf{P}}}^{- \neg E}$. As we noted in section \ref{section: Language}, when $\neg$ is not involved, the only $\GN$ formulas are those built from $\bot$ and the connectives. So no global quantification is considered.

\subsubsection{Comparison with intermediate logics}\label{section: Comparison with intermediate logics}

We see that $\CQC$ corresponds to $\SFQ_{\textup{\textbf{P}}}^{- \neg}$ under putting $\wneg \wneg$ in front. In proposition \ref{proposition: Assertibility and validity in the prevalent models}, we saw that in a prevalent model, prevalence behaves classically, and validity has a simple characterisation. We can reproduce those results in terms of theorems.
\begin{proposition}\label{proposition: neg neg A in NSFp is classical}
  (i) $A \in \CQC$ iff $\wneg \wneg A \in \SFQ_{\textup{\textbf{P}}}^{-\neg E}$. (ii) If $A \in \ST_\textup{P}[\mathcal{L}^{-E}]$, then $A \in \CQC$ implies $A \in \SFQ_{\textup{\textbf{P}}}^{-\neg E}$.
\end{proposition}
\begin{proof}
  (i) To each $W \in \langle \{r, k\}, \leq, D, J, v \rangle \in \mathcal{W}_{\textup{P}2}$ with $r < k$, associate a classical model $\phi(W) = \langle D, J, v(k) \rangle$. This correspondence $\phi$ is surjective. Prove by induction that $A$ is classically valid iff $\models_{\mathcal{W}_{\textup{P}2}}^A A$.
  (ii) Use (i).
\end{proof}
\noindent We can also involve $E$ in the scope of discussion. Let us consider $\CQC + \forall x E \subseteq \ClForm[\mathcal{L}]$.
\begin{corollary}\label{proposition: When a classical theorem is a theorem}
  (i) $A \in \CQC + \forall x E$ iff $\wneg \wneg A \in \SFQ_{\textup{\textbf{P}}}^{-\neg}$. (ii) If $A \in \ST_\textup{P}[\mathcal{L}]$, then $A \in \CQC + \forall x E$ implies $A \in \SFQ_{\textup{\textbf{P}}}^{-\neg}$.
\end{corollary}
\begin{proof}
  Similar to proposition \ref{proposition: neg neg A in NSFp is classical}. $\CQC + \forall x E$ is sound and complete with respect to the class of the classical models where $\forall x E$ is true.
\end{proof}

We denote the forcing relation of $\IQC$ by $\Vdash$. The only difference from $\models$ (as defined in definition \ref{definition: Actual verification relation models_W}) is that $l \Vdash_W B \to C$ iff for any $l' \geq l$, $l' \Vdash_W B$ implies $l' \Vdash_W C$, and quantification is always treated as global\footnote{
Indeed, then $k \Vdash_W \forall x A$ iff for all $k' \geq k$ and $d \in D$, $k' \Vdash_W A[\overline{d}/x]$. Therefore this description of $\Vdash$ works as long as we consider constant domain frames only.
}.
Since every strict finitistic model is an intuitionistic model with restrictions, it makes sense to evaluate an $A \in \ClForm[\mathcal{L}]$ (including predicate $E$) in any $W \in \mathcal{W}$ under $\Vdash$. We will write $\Vdash_W^V$ etc. similarly to $\models_W^V$ etc.

We can show intermediate logic $\HTCD = \IQC + \textup{HTQ} + \textup{CD}$ is embeddable to $\SFQ_{\textup{\textbf{P}}}^{-\neg E}$ under a translation. HTQ, the axiom we call `quantified here-and-there', is the universal closure of $A \lor (A \to B) \lor \wneg B$; and CD, the `constant domain' axiom, is $\forall x (C \lor A) \to C \lor \forall x A$, where $x$ does not occur free in $C$. It is known that $\HTCD$ is sound and complete with respect to $\mathcal{W}_{2\textup{iCD}}$, the class of the 2-node intuitionistic models with constant domains\footnote{
For a proof, see \cite[p.52]{Ono1983}. $\HTCD$ is treated under the name `$J_* + D$'.
}. 
We note that $\SFQ_{\textup{\textbf{P}}}^{-\neg E} \not \subseteq \HTCD$, since $((A \to B) \to A) \to A \notin \HTCD$; and $\HTCD \not \subseteq \SFQ_{\textup{\textbf{P}}}^{-\neg E}$, since $\exists x (P(x) \to P(x)) \notin \SFQ_{\textup{\textbf{P}}}^{-\neg E}$.\footnote{
We call the propositional version of HTQ, $A \lor (A \to B) \lor \wneg B$, `HT'. Propositional intermediate logic $\HT = \IPC + \textup{HT}$ (the logic of `here-and-there') is known to be sound and complete with respect to the class of the intuitionistic 2-node models (with the `here' and the `there' nodes). Also, $\HT$ is the strongest proper intermediate logic: i.e., it is the unique predecessor of $\CPC$ in the lattice of the propositional intermediate logics under set-theoretical inclusion. We however do not know how strong $\HTCD$ is.
}

We define an operation $^\star$ on $\Fm[\mathcal{L}^{-\neg E}]$ as follows. $P^\star := P$ for every atom. $(A \square B)^\star := A^\star \square B^\star$ where $\square \in \{ \land, \lor, \to \}$. $(\forall x A)^\star := \forall x (A^\star)$. $(\exists x A)^\star := \wneg \wneg \exists x (A^\star)$. (We consider local quantification only: cf. right before section \ref{section: Comparison with intermediate logics}.) Intuitively, $^\star$ only weakens local existential quantification by removing the requirement of the existing witness. It attains the effect of putting $\neg \neg$ inside by putting $\wneg \wneg$ in front, via $\wneg \wneg \exists x A \dashv \vdash_{\NSFp} \neg \neg \exists x A \dashv \vdash_{\NSFp} \exists x \neg \neg A$. We note that plainly $(A[c/x])^\star = (A^\star)[c/x]$ for all $c \in \ClTerm[\mathcal{L}]$, if no quantificational $QxB$ occurs in $A$. 
\begin{lemma}\label{lemma: Lemma for the star operation}
  For any $W \in \mathcal{W}_{\textup{P}}$, (i) $\models_W^A A$ iff $\models_W^A A^\star$, and (ii) $\models_W^V A$ implies $\models_W^V A^\star$.
\end{lemma}
\begin{proof}
  Induction. 
\end{proof}

\begin{proposition}\label{proposition: HTQ and NSFp}
  (i) $\Vdash_{\mathcal{W}_{\textup{P}2}}^{V} A$ implies $\models_{\mathcal{W}_{\textup{P}2}}^V A^\star$. (ii) $[\HTCD]^\star \subseteq \SFQ_{\textup{\textbf{P}}}^{-\neg E}$.
\end{proposition}
\begin{proof}
  (i) Fix one $W = \langle \{r, k\}, \leq, D, J, v \rangle \in \mathcal{W}_{\textup{P}2}$ with $r < k$. Then, prove that $k \Vdash A$ iff $k \models A$, and $r \Vdash A$ implies $r \models A^\star$, using lemma \ref{lemma: Lemma for the star operation}. (ii) Use $\vdash_{\HTCD} A$ iff $\Vdash_{\mathcal{W}_{2\textup{iCD}}}^V A$, $\mathcal{W}_{\textup{P}2} \subseteq \mathcal{W}_{2\textup{iCD}}$ and (i). 
\end{proof}
\noindent 
We take this as a quantificational version of \cite{Yamada2023}'s proposition 2.26.



\subsubsection{Prevalent models as intuitionistic nodes}\label{section: A prevalent model as an intuitionistic node}
We will see a strong similarity between $\SFQ_{\textup{\textbf{P}}}^{-\neg}$ and $\IQC$. We will introduce a `generation structure' as a collection of finite prevalent models equipped with an order relation, and show that it can be viewed as an intuitionistic model. (For any $A \in \ClForm[\mathcal{L}^{- \neg}]$, plainly  by completeness and proposition \ref{proposition: Strong contraction},  $A \in \SFQ_{\textup{\textbf{P}}}^{-\neg}$ iff $A$ is valid in all finite prevalent models). We will provide an estimation of the totality $Th(\mathcal{G})$ of the formulas valid in the sense of the generation structures: $[\IQC]^\star \subsetneq Th(\mathcal{G})^{-\neg E} \subsetneq \IQC$ (proposition \ref{proposition: T Th G and IQC}). We end our investigation by showing a class of generation structures whose validity coincides with $\IQC$ (proposition \ref{corollary: Th GC is IQC}). We take this result as a quantificational version of \cite{Yamada2023}'s proposition 2.32, i.e., the formalised identity principle.

We define the `generation semantics' using the notions of the strict finitistic semantics. Let $\mathcal{W}_{\textup{P} \fin} \subseteq \mathcal{W}_{\textup{P}}$ be the subclass of the finite models, and $\mathcal{U} \, (\neq \emptyset) \subseteq \mathcal{W}_{\textup{P} \fin}$ be at most countable. We assume in general each $W \in \mathcal{U}$ consists of $\langle K_W, \leq_W, D_W, J_W, v_W \rangle$, where $J_W = \langle J_W^1, J_W^2 \rangle$. A binary relation $\preceq \subseteq \mathcal{U}^2$ is a \textit{generation order} on $\mathcal{U}$ if $W \preceq W'$ iff the following hold\footnote{
Not to be confused with $\preccurlyeq$ for the well-ordering on the canonical model of $\SFQ$ in section \ref{section: The proof of completeness}, definition \ref{definition: Well-ordering preceq}.
}. 
(i) $K_W \subseteq K_{W'}$. (ii) For all $k, k' \in K_W$, $k \leq_W k'$ iff $k \leq_{W'} k'$. (iii) $D_W \subseteq D_{W'}$. (iv) $J_W^1(c) = J_{W'}^1(c)$ for all $c \in \ClTerm[\mathcal{L}(D_W)]$, and $J_W^2(f) \subseteq J_{W'}^2 (f)$ for all $f \in \Func$. (v) For all $k \in K_W$ and $P \in \Pred[\mathcal{L}]$, $P^{v_W(k)} \subseteq P^{v_{W'}(k)}$. Plainly, a generation order is a partial order on $\mathcal{U}$. $G = \langle \mathcal{U}, \preceq \rangle$ is a \textit{generation structure} (\textit{g-structure}) if $\preceq$ is a generation order. We will only be considering rooted tree-like generation structures with at most countable branchings and height at most $\omega$\footnote{
These are the same conditions on the intuitionistic frames (cf. section \ref{section: Semantics}).
}.
We denote the class of all generation structures by $\mathcal{G}$. We may write $R_G$ for a generation structure $G$'s root, and $r_G$ for $R_G$'s root. Also, we let $K_G = \{ \langle W, k \rangle \mid W \in \mathcal{U} \land k \in K_W \}$ and $D_G = \bigcup_{W \in \mathcal{U}} D_W$.

Conceptually, a $G = \langle \mathcal{U}, \preceq \rangle \in \mathcal{G}$ models all possible constructions and verificatory developments of a cognitive (possibly communal) agent $\mathcal{U}$ through generations. $\prec$ is the `earlier-than' relation on the generations. If $W \preceq W'$, then $W'$ has more objects constructible and more facts verifiable. $W'$ contains all nodes of $W$ and their relevant information: a later generation has more cognitive power and keeps all the old methods of construction and verification.

Given a $G = \langle \mathcal{U}, \preceq \rangle \in \mathcal{G}$, we induce, from the valuations of the elements of $\mathcal{U}$, a valuation $v_G$ that assigns to each $P \in \Pred[\mathcal{L}]$ its extension at each `node-in-generation' $\langle W, k \rangle \in K_G$: $P^{v_G(W, k)} := P^{v_W(k)}$. Then $v_G$ persists in both $G$ and all $W$.
\begin{proposition}
  If $W \preceq W'$ and $k \leq_{W'} k'$, then $P^{v_G(W, k)} \subseteq P^{v_G(W', k')}$.
\end{proposition}
\begin{proof}
  By persistence of $v_W$ in each $W$. 
\end{proof}
\noindent We consider a new forcing relation between $K_G$ and $\Fm[\mathcal{L}^{-\neg}(D_W)]$.
\begin{definition}[Generation forcing relation $\Vvdash_G$]
  Let $\langle W, k \rangle \in K_G$. $W, k \Vvdash_{G} \top$ and $W, k \not \Vvdash_G \bot$. For all $P \in \Pred[\mathcal{L}]$ and $c \in \ClTerm[\mathcal{L}(D_W)]$, $W, k \Vvdash_G P(c)$ iff $\langle J_W(c) \rangle \in P^{v_G(W, k)}$. For all $A, B \in \ClForm[\mathcal{L}^{-\neg}(D_W)]$,
  \begin{enumerate}
    \item $W, k \Vvdash_G A \land B$ iff $W, k \Vvdash_G A$ and $W, k \Vvdash_G B$;
    \item $W, k \Vvdash_G A \lor B$ iff $W, k \Vvdash_G A$ or $W, k \Vvdash_G B$;
    \item $W, k \Vvdash_G A \to B$ iff for any $W' \succeq W$ and $k' \, (\in K_{W'}) \geq_{W'} k$, if $W', k' \Vvdash_G A$, then there is a $k'' \, (\in K_{W'}) \geq_{W'} k'$ such that $W', k'' \Vvdash_G B$;
    \item $W, k \Vvdash_G \forall x A$ iff for any $W' \succeq W$ and $d \in D_{W'}$, $W', k \Vvdash_G E(\overline{d}) \to A[\overline{d}/x]$; and
    \item $W, k \Vvdash_G \exists x A$ iff there is a $d \in D_W$ such that $W, k \Vvdash_G E(\overline{d}) \land A[\overline{d}/x]$.
  \end{enumerate}
  For an open $A$, $W, k \Vvdash_G A$ if $W, k \Vvdash_G A^*$, where $A^*$ is the universal closure.
\end{definition}
\noindent We say $A \in \Fm[\mathcal{L}^{-\neg}]$ is \textit{valid in a $G \in \mathcal{G}$} if $W, k \Vvdash_G A$ for all $\langle W, k \rangle \in K_G$. We write $\Vvdash_G^V A$ for this; and $\Vvdash_{\mathcal{G}}^V A$ for that $\Vvdash_G^V A$ for all $G \in \mathcal{G}$.

$\Vvdash_G$ thus defined formalises the agent's actual verifiability through their generations from our perspective. $W, k \Vvdash_G A \to B$ means that $B$ comes soon after $A$ in the sense that it comes within the same generation. $W, k \Vvdash_G \wneg A$ is local negation that is forward-looking: it means that $A$ is practically unverifiable no matter how the agent grows from $\langle W, k \rangle$ onwards.

Generation structures carry over the semantic properties of prevalent models. The strictness holds, i.e., $W, k \Vvdash_G P(c)$ implies $W, k \Vvdash_G E(c')$ for all subterms $c'$ of $c$. Also, $\Vvdash_G$ persists in $G$ and all $W$ just as $v_G$. Therefore $\Vvdash_G^V A$ iff $R_G, r_G \Vvdash_G A$. We see that each generation $W$ has both the object prevalence and the formula prevalence inside.
\begin{proposition}[Prevalence in generation]
  (i) For all $d \in D_W$ and $k \in K_W$, there is a $k' \, (\in K_W) \geq k$ such that $W, k' \Vvdash_G E(\overline{d})$. (ii) If $W, k \Vvdash_G A$ for some $k \in K_W$, then for any $l \in K_W$, there is an $l' \, (\in K_W) \geq_W l$ such that $W, l' \Vvdash_G A$.
\end{proposition}
\begin{proof}
  (i) By object prevalence of $W \in \mathcal{W}_{\textup{P}}$. (ii) Induction on $A$. Use (i) and persistence of $\Vvdash_G$. 
\end{proof}
\noindent The following states (i) that accordingly, we could have defined the ($\forall$)-clause by the global condition, i.e. with $\top$ in place of $E(\overline{d})$, and (ii) that existential quantification on the level of assertibility in a generation can be treated as if it were global.
\begin{corollary}\label{corollary: The distinction is lost for forall and exists, in a sense}
  (i) For any $\langle W, k \rangle \in K_G$, $W, k \Vvdash_G \forall x A$ iff for any $W' \succeq W$ and $d \in D_{W'}$, $W', k \Vvdash_G \top \to A(\overline{d})$. (ii) For any $W \in \mathcal{U}$, $W, k \Vvdash_G \exists x A$ for some $k \in K_W$ iff $W, k \Vvdash_G A(\overline{d})$ for some $d \in D_W$ and $k \in K_W$.
\end{corollary}
\begin{proof}
  By prevalence in generation. 
\end{proof}
\noindent Also, implication and universal quantification in a generation behave just as in a prevalent model (cf. proposition \ref{proposition: Assertibility and validity in the prevalent models}).
\begin{corollary}\label{corollary: Assertibility of implication in g-structure is validity}
  (i) $W, r_G \Vvdash_G A \to B$ iff $W, k \Vvdash_G A \to B$ for some $k \in K_W$. (ii) $W, r_G \Vvdash_G \forall x A$ iff $W, k \Vvdash_G \forall x A$ for some $k \in K_W$.
\end{corollary}
\begin{proof}
  (i) Follows from prevalence in generation. (ii) Use (i). 
\end{proof}

Now, we will establish that a generation structure induces a corresponding intuitionistic model, and (partly) vice versa. Let $\mathcal{I}$ denote the class of the intuitionistic models of $\mathcal{L}^{-\neg}$ with the finite verification condition (but with no condition for $E$). We note that plainly, for any $A \in \ClForm[\mathcal{L}^{-\neg E}]$, $\vdash_{\IQC} A$ iff $\Vdash_I^V A$ for all $I \in \mathcal{I}$. 

First, we see that a $G = \langle \mathcal{U}, \preceq \rangle \in \mathcal{G}$ gives rise to an $I \in \mathcal{I}$. Given $G$, we construct a tuple $I_G = \langle \mathcal{U}, \preceq, \mathcal{D}, \mathcal{J}, v \rangle$ from the assertible formulas in each $W \in \mathcal{U}$ as follows. (i) Define a function $\mathcal{D}: \mathcal{U} \to \mathcal{P}(D_G)$ by $\mathcal{D}(W) := D_W$. (ii) Let $\mathcal{J} = \langle J_W \rangle_{W \in \mathcal{U}}$. (iii) Define a valuation $v$ by $P^{v(W)} := \bigcup_{k \in K_W} P^{v_W(k)}$ for each $P \in \Pred[\mathcal{L}]$ (including $E$). Then plainly $v$ persists in $G$, and $I_G$ is an intuitionistic model. Also, since every $W$ is finite, $I_G$ satisfies the finite verification condition. Therefore $I_G \in \mathcal{I}$. So the intuitionistic forcing relation $W \Vdash_{I_G} A$ (cf. section \ref{section: Comparison with intermediate logics}) makes sense for each $W \in \mathcal{U}$ and $A \in \Fm[\mathcal{L}^{-\neg}(D(W))]$; and $\Vdash_{I_G}$ persists in $G$. By $\Vdash_{I_G}^V A$, we mean that $W \Vdash_{I_G} A$ for all $W \in \mathcal{U}$. Then $\Vdash_{I_G}^V A$ iff $R_G \Vdash_{I_G} A$. We note that $\Vdash_{I_G} \forall x E$ holds for any $G \in \mathcal{G}$, by object prevalence of each $W \in \mathcal{U}$. $I_G$ corresponds to $G$ in the following sense.
\begin{proposition}\label{proposition: G to I transformation}
  $W \Vdash_{I_G} A$ iff there is a $k \in K_W$ such that $W, k \Vvdash_G A$.
\end{proposition}
\begin{proof}
  Induction. Use prevalence in generation. 
\end{proof}
\noindent Thus the intuitionistic forcing relation on $I_G$ stands for the assertibility in each of the generations that comprise $G$.

Conversely, let $I = \langle U^*, \preceq^*, D^*, J^*, v^* \rangle \in \mathcal{I}$, where $J^* = \langle \langle J_U^{*1}, J_U^{*2} \rangle \rangle_{U \in U^*}$. We will define a generation structure $G_I = \langle \mathcal{U}, \preceq \rangle$. Each $U \in U^*$ will be the greatest node of a linear prevalent model $W_U$. Formally, for each $U \in U^*$, we define $W_U = \langle K_U, \leq_U, D_U, J_U, v_U \rangle$ as follows. $K_U := \{ U' \in U^* \mid U' \preceq^* U \}$. $U' \leq_U U''$ iff $U' \preceq^* U''$. $D_U := D^*(U)$. $J_U := \langle J_U^{*1}, J_U^{*2} \rangle$. For all $U' \in K_U$, $E^{v_U(U')} := D^*(U')$, and $P^{v_U(U')} := P^{v^* (U')}$ for each $P \, (\neq E)$. (Therefore $E^{v_U(U')} = E^{v^* (U')}$ may fail.) Then $W_U$ satisfies the strictness condition and the finite verification condition, and $W_U \in \mathcal{W}_{\textup{P} \fin}$. 
We define $\mathcal{U} := \{ W_U \mid U \in U^* \}$; $\preceq \, \subseteq \mathcal{U}^2$ by $W_{U_1} \preceq W_{U_2}$ iff $U_1 \preceq^* U_2$. Then $G_I = \langle \mathcal{U}, \preceq \rangle \in \mathcal{G}$.

We note that for each $U \in U^*$, identical are (1) $U$'s domain $D^*(U)$ as a node in $I$, (2) the constant domain $D_U$ of $W_U$ and (3) $E$'s extension $E^{v_{G_I}(W_U, U)}$ at the node-in-generation $\langle W_U, U \rangle$. Therefore, for all $c \in \ClTerm[\mathcal{L}(D^*(U))] = \ClTerm[\mathcal{L}(D_U)]$, $W_U, U \Vvdash_{G_I} E(c)$. This implies that the forcing condition of $\exists$ at the maximum node of a generation is always intuitionistic: $W_U, U \Vvdash_{G_I} \exists x A$ iff there is a $d \in D_U$ such that $W_U, U \Vvdash_{G_I} A[\overline{d}/x]$. 

$G_I$ corresponds to $I$ in the following sense.
\begin{proposition}\label{proposition: J to G transformation}
  For all $U \in U^*$ and $A \in \Fm[\mathcal{L}^{-\neg E}(D^*(U))]$, $U \Vdash_I A$ iff $W_U, U \Vvdash_{G_I} A$.
\end{proposition}
\begin{proof}
  Induction. 
  ($\to$) ($\implies$) Suppose $U \Vdash_I B \to C$, and $W_{U'}, U'' \Vvdash_{G_I} B$ with $U \preceq^* U'' \preceq^* U'$. Then $W_{U'}, U' \Vvdash_{G_I} B$ by persistence. So $U' \Vdash_I B$, and hence $U' \Vdash_I C$. Therefore $W_{U'}, U' \Vvdash_{G_I} C$. ($\impliedby$) Suppose $W_U, U \Vvdash_{G_I} B \to C$, $U \preceq^* U'$ and $U' \Vdash_I B$. Then $W_{U'}, U' \Vvdash_{G_I} B$. Since $U'$ is the maximum of $\langle K_{U'}, \leq_{U'} \rangle$, $W_{U'}, U' \Vvdash_{G_I} C$ by the supposition.
  
  ($\forall$) ($\implies$) Suppose $U \Vdash_I \forall x B$, $U \preceq^* U'$ and $d \in D_{U'}$. Then $U' \Vdash_I B[\overline{d}/x]$. So $W_{U'}, U' \Vvdash_{G_I} B[\overline{d}/x]$, implying $W_{U'}, U \Vvdash_{G_I} \top \to B[\overline{d}/x]$. ($\impliedby$) Suppose $W_{U}, U \Vvdash_{G_I} \forall x B$, $U \preceq^* U'$ and $d \in D^*(U')$. Since $W_{U'}, U \Vvdash_{G_I} \top \to B[\overline{d}/x]$, and $U'$ is the maximum of $\langle K_{U'}, \leq_{U'} \rangle$, we have $W_{U'}, U' \Vvdash_{G_I} B[\overline{d}/x]$, implying $U' \Vdash_I B[\overline{d}/x]$.
\end{proof}

Now we provide an estimation of the validity in all the generation structures. By $Th(\mathcal{G})^{-\neg}$, we mean the theory of $\mathcal{G}$ in $\mathcal{L}^{- \neg}$, i.e., $\{ A \in \ClForm[\mathcal{L}^{-\neg}] \mid \Vvdash_{\mathcal{G}}^V A \}$; also we let $Th(I_\mathcal{G})^{-\neg} = \{ A \in \ClForm[\mathcal{L}^{-\neg}] \mid \forall G \in \mathcal{G} (\Vdash_{I_G}^V A) \}$. Similarly, by $Th(\mathcal{G})^{-\neg E}$ and $Th(I_\mathcal{G})^{-\neg E}$, we mean $Th(\mathcal{G})^{-\neg}$ and $Th(I_\mathcal{G})^{-\neg}$ with $\mathcal{L}^{-\neg E}$ in place of $\mathcal{L}^{-\neg}$, respectively. In what follows, we appeal to the disjunction property and the existence property of $\IQC$\footnote{
See e.g. \cite[pp.266-7]{vanDalen1986}.}.
\begin{proposition}\label{proposition: T Th G and IQC}
  $[\IQC]^\star \subsetneq Th(\mathcal{G})^{-\neg E} \subsetneq Th(I_\mathcal{G})^{-\neg E} = \IQC$. 
\end{proposition}
\begin{proof}
  (i) To prove $[\IQC]^\star \subseteq Th(\mathcal{G})^{-\neg E}$, establish by induction that for all $A \in \ClForm[\mathcal{L}^{-\neg E}]$, $\vdash_{\IQC} A$ implies $\Vvdash_{\mathcal{G}}^V A^\star$. The basis and ($\land$) are trivial. ($\lor$) By disjunction property of $\IQC$. ($\to$) does not depend on the induction hypothesis. Suppose $\Vdash_I^V B \to C$ for all $I \in \mathcal{I}$. Let $G = { \langle \mathcal{U}, \preceq \rangle } \in \mathcal{G}$. To prove $R_G, r_G \Vvdash_G (B \to C)^\star$, it suffices to show that $R_G, k \Vvdash_G B \to C$ for some $k \in K_{R_G}$, by lemma \ref
{lemma: Lemma for the star operation} and corollary \ref{corollary: Assertibility of implication in g-structure is validity} (i). Here we have $R_G \Vdash_{I_G} B \to C$ by the supposition. Apply proposition \ref{proposition: G to I transformation}. ($\forall$) Similar to ($\to$). ($\exists$) By existence property, $\vdash_{\IQC} B[t/x]$ for some $t \in \Term[\mathcal{L}]$. If $t$ is open, then $\vdash_{\IQC} \forall x B$, and the case reduces to ($\forall$). If $t$ is closed, then use the induction hypothesis and the object prevalence in the root node. 
The converse fails, since $\exists x (P(x) \to P(x)) \lor (A \to A) \in Th(\mathcal{G})^{-\neg E} \backslash [\IQC]^\star$.
  
  (ii) $Th(\mathcal{G})^{-\neg E} \subseteq Th(I_\mathcal{G})^{-\neg E}$ follows from proposition \ref{proposition: G to I transformation}. For the failure of the converse, consider $\exists x (P(x) \to P(x)) \in \IQC$. 
  
  (iii) To prove $Th(I_\mathcal{G})^{-\neg E} \subseteq \IQC$, let $H \in \mathcal{I}$, and $R$ be $H$'s root. Confirm that $\Vdash_H^V A$ iff $R \Vdash_H A$ iff $W_R, R \Vvdash_{G_H} A$ iff $W_R \Vdash_{I_{G_H}} A$. $A \in Th(I_\mathcal{G})^{-\neg E}$ implies the last. The converse is trivial.
\end{proof}

This estimation can be sharpened by a restriction. Indeed, there is a class of generation structures whose validity precisely coincides with $\IQC$. We say a $W \in \mathcal{W}$ is \textit{postconstructive} if $\models_{W}^V E(c)$ for all $c \in \ClTerm[\mathcal{L}]$, in the sense that all the objects prescribed by the default language have already been constructed at the beginning. We write $\mathcal{W}_{\textup{C}} \subseteq \mathcal{W}_{\textup{P} \fin}$ for the subclass of such models. We note that this is an `opposite' notion of the preconstructive models (cf. proposition \ref{proposition: Reduction of Wp to Wp2ets}).

We say a $G = \langle \mathcal{U}, \preceq \rangle \in \mathcal{G}$ is \textit{postconstructive} if $\mathcal{U} \subseteq \mathcal{W}_{\textup{C}}$. We write $\mathcal{G}_{\textup{C}}$ for the class of all postconstructive g-structures. Also, let $Th(\mathcal{G}_{\textup{C}})^{-\neg} = \{ A \in \ClForm[\mathcal{L}^{-\neg}] \mid \Vvdash_{\mathcal{G}_{\textup{C}}}^V A \}$; and similar for $Th(I_{\mathcal{G}_{\textup{C}}})^{-\neg}$, $Th(\mathcal{G}_{\textup{C}})^{-\neg E}$ and $Th(I_{\mathcal{G}_{\textup{C}}})^{-\neg E}$. We will show that by restricting to $\mathcal{G}_\textup{C}$, proposition \ref{proposition: T Th G and IQC} is tightened up to $Th(\mathcal{G}_\textup{C})^{-\neg E} = \IQC$.

We note that $G_I \in \mathcal{G}_{\textup{C}}$ for all $I = \langle U^*, \preceq^*, D^*, J^*, v^* \rangle \in \mathcal{I}$, since by definition $J_R^{*1} [\ClTerm[\mathcal{L}]] \subseteq D^*(R) = E^{v_{G_I}(W_R, R)}$, where $R$ is $I$'s root (cf. right before proposition \ref{proposition: J to G transformation}). Also, let $\mathcal{I}^E \subseteq \mathcal{I}$ be the subclass with $\forall x E$ valid. Then, $\IQC + \forall x E$ is sound and complete with respect to $\mathcal{I}^E$, and $I_G \in \mathcal{I}^E$ for all $G \in \mathcal{G}$.
\begin{proposition}\label{corollary: Th GC is IQC}
  (i) $Th(\mathcal{G}_\textup{C})^{-\neg E} = \IQC$. (ii) $Th(\mathcal{G}_\textup{C})^{-\neg} = \IQC + \forall x E$.
\end{proposition}
\begin{proof}
  (i) Prove $\IQC \subseteq Th(\mathcal{G}_\textup{C})^{-\neg E} \subseteq Th(I_{\mathcal{G}_\textup{C}})^{-\neg E} \subseteq \IQC$ similarly to proposition \ref{proposition: T Th G and IQC}. 
  (ii) It holds that $\IQC + \forall x E \subseteq Th(\mathcal{G}_\textup{C})^{-\neg} \subseteq Th(I_{\mathcal{G}_\textup{C}})^{-\neg} \subseteq \IQC + \forall x E$. To prove the first inclusion, use that $\IQC + \forall x E$ has the disjunction property and the existence property, since $\forall x E$ is a Harrop formula. 
\end{proof}

\section{Ending remarks}\label{section: Ending remarks: Further topic}
In this article, we presented a classical rendition of Wright's strict finitistic notions, and a reconstruction of first-order logic of his strict finitism. Section \ref{section: Methods: Our rendition of Wright's semantic} explained Wright's original semantic system and the notions central to it, and explicitly stated, together with our conceptual assumptions, how we render them in order to reconstruct the system. In the first half (section \ref{section: Strict finitistic first-order logic}) of the mathematical part, we treated the logic in the general setting, and provided a Kripke semantics and a proof system that are sound and complete. 
In the second half (section \ref{section: Theory of prevalence}), we investigated the logic in the prevalent setting. To us, the formalisation theorems (\cite[proposition 2.32]{Yamada2023} and our proposition \ref{corollary: Th GC is IQC}) of the bridging principle (\ref{principle: Conceptual identity}) are of special import, since they appear to formally show in what sense the notion of verifiability in principle is an `extension' of that of verifiability in practice, under our rendition of Wright's notions. If the rendition is correct, the theorems may indicate that Wright's sketch of the semantics aptly grasped what verifiability in practice is, and that Yamada \cite{Yamada2023}'s and our classical reconstruction are a step in the right direction towards formalising strict finitism. 

The key concept there was prevalence. All thanks to it, the histories represented by a strict finitistic model become all homogeneous. The assertibility becomes classical, and therefore we can view a partially ordered set of models as an intuitionistic model whose nodes consist of the sets of the assertible formulas (cf. proposition \ref{proposition: G to I transformation}). Yamada \cite[section 3]{Yamada2023} suspected that the prevalence was responsible for the correspondence between $\SF$ and $\IPC$. We saw that the prevalence does connect $\SFQ_{\textup{\textbf{P}}}^{- \neg}$ and $\IQC$ to some extent (proposition \ref{proposition: T Th G and IQC}). 
However, proposition \ref{corollary: Th GC is IQC} shows that for the two logics to match, the models considered need to be postconstructive, i.e., all the objects prescribed by the default language are given at the root node. Since the validity in $\SFQp$ is captured by the class $\mathcal{W}_{\textup{P}2 \emptyset}$ of the preconstructive models (cf. proposition \ref{proposition: Reduction of Wp to Wp2ets}), this formally corroborates that strict finitism pays attention not only to actual verification, but also to actual construction of objects, whereas intuitionism does neither.

The `time-compressing' nature of intuitionistic reasoning may, this way, be seen more clearly. Intuitionistic implication $A \to B$ is strict finitistic implication with the time-gap compressed. $B$ is forced at the same node as $A$, no matter how late it may come in practice: if $B$ will be brought about in principle, then it is as if it is verified already. Intuitionistic logic assumes that all the default objects are constructed at the beginning. Construction according to intuitionism appears to be as if, since all objects will be constructed in principle, we can assume they are already constructed.

We must leave one obvious further task untouched, i.e. the investigation of the formal system of strict finitistic arithmetic. If this article's logic formalises strict finitistic reasoning plausibly enough, then what system would come out when we add arithmetical axioms, or what axioms should we need in order to reflect the strict finitistic assumptions of numbers? Since the standpoint is meant to be a view of numbers first and foremost, it is imperative that these questions are answered. 


\section*{Acknowledgements}
\begin{itemize}
  \item Rosalie Iemhoff has my utmost thanks. Any of my mathematical works will eternally be indebted to her tutelage. I thank an anonymous referee for pointing out the lack of conceptual underpinnings; and Amirhossein Akbar Tabatabai for discussions on detailed matters. The contents of this article were previously presented at (i) the Workshop on Proofs and Formalization in Logic, Mathematics and Philosophy, (ii) the 4th International Workshop on Proof Theory and Its Applications, (iii) the Masterclass in the Philosophy of Mathematical Practices 2023 with Jean Paul Van Bendegem and (iv) the 5th International Workshop on Proof Theory and Its Applications. This study was conducted under the doctoral supervision by Professor Rosalie Iemhoff at Utrecht University.
\end{itemize}

\end{document}